\documentclass[twoside,12pt,leqno]{article}
\usepackage{amsmath, amsthm, amsfonts, amssymb, color}
\usepackage{mathrsfs}
\usepackage{fancyhdr}

\newcommand\dela[1]{}

\newcommand{\be}{\begin{eqnarray}}
\newcommand{\ee}{\end{eqnarray}}
\newcommand{\ce}{\begin{eqnarray*}}
\newcommand{\de}{\end{eqnarray*}}

\newtheorem{thm}{Theorem}[section]
\newtheorem{cor}[thm]{Corollary}
\newtheorem{lem}[thm]{Lemma}

\newtheorem{prp}[thm]{Proposition}

\newtheorem{exa}[thm]{Example}

\theoremstyle{definition}
\newtheorem{defn}{Definition}[section]

\definecolor{wco}{rgb}{0.5,0.2,0.3}

\numberwithin{equation}{section}
\newtheorem{rem}[thm]{Remark}

\def\R{\mathbb R}   
\def\N{\mathbb N}  
   
\def\<{\langle} \def\>{\rangle}  
    
\def\d{\text{\rm{d}}}   
\def\E{\mathbb E}  
\def\beg{\begin} \def\beq{\begin{equation}}

 \def\P{\mathbb P} 
 
 \def\ee{\varepsilon}

\def\[{{\Big[}}
\def\]{{\Big]}}

\def\({{\Big(}}
\def\){{\Big)}}

\allowdisplaybreaks[1]

\begin{document}

\title{{\bf Strong solutions for SPDE  with locally monotone coefficients driven by  L\'{e}vy noise}
\footnote{Supported in part by NSFC (No.11201234), NSF of Jiangsu Higher
Education Institutions (No.12KJB110014), the
PAPD of Jiangsu Higher Education Institutions, the DFG through SFB-701 and IGK 1132.} }

\author{ {\bf Zdzis\l{}aw Brze\'{z}niak~$^a$,  Wei Liu $^{b,c}$\footnote{Corresponding author: weiliu@math.uni-bielefeld.de} ,
 Jiahui Zhu $^{d}$
}\\
 \footnotesize{$a.$ Department of Mathematics, University of York, YO105DD  York, UK}\\
 \footnotesize{  $b.$ School of Mathematical Sciences, Jiangsu Normal University, 221116 Xuzhou, China}\\
  \footnotesize{  $c.$  Fakult\"at f\"ur Mathematik, Universit\"at Bielefeld,
D-33501 Bielefeld, Germany}\\
\footnotesize{  $d.$ School of Finance, Zhejiang University of Finance and Economics, 310018 Hangzhou, China}\\
}

\date{ }
\maketitle

\begin{abstract} Motivated by applications to a manifold of semilinear and quasilinear
 stochastic partial differential equations (SPDEs) we  establish  the existence and uniqueness of
strong solutions to coercive and  locally monotone SPDEs driven by
L\'{e}vy processes. We illustrate  the  main result of our paper by
showing how it can be applied  to various types of SPDEs such as
stochastic reaction-diffusion equations, stochastic Burgers type equations, stochastic 2D
hydrodynamical systems and stochastic equations of non-Newtonian fluids, which generalize many existing results in the literature.
\end{abstract}


\bigbreak

\section{Introduction and Main Results}

In recent years,  Stochastic Partial Differential Equations (SPDEs)
driven by jump type noises such as L\'{e}vy-type or Poisson-type
perturbations  have become extremely popular for modeling financial,
physical and biological phenomena. In some circumstances, purely
Brownian motion perturbation has many imperfections while capturing
some large moves and unpredictable events. L\'{e}vy-type
perturbations come to the stage to reproduce the performance of
those natural phenomena in some real world models.
  The existence and uniqueness of solutions for SPDEs driven by jump type noises
   has  already  been intensively investigated  by many authors,
see e.g. Kallianpur and Xiong \cite{[Kal+X]},   Albeverio et al
\cite{AWZ},  Mueller et al \cite{M98,mueller}, Applebaum and Wu
\cite{Ap+Wu}, Mytnik \cite{Mytnik_2002}, Truman and Wu
\cite{Tru+Wu},  Hausenblas \cite{Hau-1,Hau-2},  Mandrekar and
R\"{u}diger \cite{MR06},
 R\"{o}ckner and Zhang \cite{RZ07},
  Dong et al \cite{DX1,Dong_2009,DXZ}, Marinelli and R\"{o}ckner \cite{MR10},  Bo et al \cite{BSW},
 Brze\'{z}niak et al \cite{Brz+Haus_2009,Brz+Hau+Zhu_2011,Brz+Zhu_2010} and
 the recent monograph by Peszat and Zabczyk \cite{PZ}.  The last reference can also be used  for more detailed expositions and references.

In this paper, we aim to establish a framework in which one can
treat a large number of  SPDEs driven by L\'{e}vy type noises
including  stochastic reaction-diffusion equations, stochastic Burgers type
equations, stochastic 2D Navier-Stokes equations and stochastic equations of non-Newtonian fluids etc.
 The line of investigation proposed in
this paper   began with the celebrated  works by
 Pardoux \cite{Par75} and   Krylov and Rozovskii \cite{KR79},  and later it was further  developed  by many authors,
  see e.g. Gy\"{o}ngy and  Krylov  \cite{[G+K]},  Gy\"{o}ngy \cite{G}.
Ren et al  \cite{RRW}, R\"{o}ckner and Wang \cite{RW}  and Zhang
\cite{Zh08}. Roughly speaking, for stochastic equations in finite
dimensional spaces, the existence and uniqueness result was obtained
under the local monotonicity assumption for the coefficients, see
\cite{KR79} for SDEs driven by Brownian motion and \cite{[G+K]} for
SDEs driven by (possibly discontinuous) locally square integrable
martingales. However, concerning the existence and uniqueness of
strong solutions to SPDEs in infinite dimensional spaces
 driven by Wiener processes or local martingales, all results were  established for the \textbf{globally} monotone coefficients SPDE
 (cf. \cite{KR79,G,RRW,Zh08}).

Recently, the classical variational  framework has been extended by
the second named author  and R\"{o}ckner in \cite{Liu+Roc}  for SPDE
driven by Wiener process in Hilbert space with  locally monotone
coefficients.
In \cite{Liu+Roc}  the authors showed that the local
monotonicity method first used by Menaldi and Sritharan \cite{MS02}
for  stochastic 2D Navier-Stokes equations (and later used by
Sritharan and Sundar \cite{Srith+S_2006}, Chueshov and Millet
\cite{CM10} for various stochastic equations of hydrodynamics) can
be generalized to such an extent that the extended variational
framework is applicable to all the equations investigated in
\cite{KR79,PR07,MS02,Srith+S_2006,CM10}.

On the other hand, there are not many papers studying non-Lipschitz
SPDEs driven by L\'{e}vy type noises with small jumps. The first and
third named author proved in
 \cite{Brz+Zab_2009}    the existence and uniqueness of solutions to stochastic nonlinear beam equations driven by L\'{e}vy type noises.
 They together with Hausenblas extended in \cite{Brz+Hau+Zhu_2011} (see also \cite{Dong_2009})
the work of Menaldi and Sritharan by showing that their method
yields the existence and uniqueness of solutions to stochastic 2D
Navier Stokes equations driven by a L\'{e}vy type noise. There is
also the work of the first named author and Hausenblas
\cite{Brz+Haus_2009} in which  by means of generalized compactness
method   the existence of solutions to stochastic reaction diffusion
equations driven by a L\'{e}vy type noise was investigated.
 What we do in the present paper is to confirm
the natural conjecture that the framework in \cite{Liu+Roc} works
not only for locally monotone SPDEs driven by  multiplicative
Gaussian noise but also by multiplicative L\'{e}vy type noise (see
Remark 1.3). However, we should point out that our results are not
applicable to evolution equations with general space time white
noise, see for instance \cite{Brz+Zab_2009,Brz+Haus_2009} and the
references therein. The reason is  that the solutions of SPDEs with
general
 space time white noise are not regular enough to fit in the variational framework.

The main contribution of this work is that we establish a unified
framework for a large class of semilinear and quasilinear SPDE
driven by  general L\'{e}vy noises, which generalizes many previous
works  \cite{Par75,KR79,G,Liu+Roc}. The main result is  applicable
to various types of concrete examples such as stochastic 2D
Navier-Stokes equations, stochastic magneto-hydrodynamic equations,
the Boussinesq model for the B\'{e}nard convection, 2D magnetic
B\'{e}nard problem, stochastic 3D Leray-$\alpha$ model (cf. Remark \ref{rem4.1}) and stochastic
equations of non-Newtonian fluids  (see
Section 4 for details). Hence it also recovers and improves many
known results in the literature, see for instance
\cite{DX1,MS02,RZ07,Dong_2009,CM10,Deugoue+Sango_2010,Brz+Hau+Zhu_2011}.
 In a recent work \cite{Brz+Zhu_2010} by the first and third named author, a type of stochastic nonlinear beam equations with Poisson-type noises was
studied and the existence and uniqueness of  solutions  was
established by following a natural route of constructing a local
mild solution and proving, with the   help of the Khasminski test,
that this solution is a global one. In contrast to
\cite{Brz+Zhu_2010},  the approach used in this paper is different.
We will follow the lines in \cite{Brz+Hau+Zhu_2011,Liu+Roc} and  the
technique involves the use of the Galerkin approximation,  local
monotonicity arguments but not, as opposed to \cite{Brz+Haus_2009},
compactness argument.
 We shall use the result from \cite{[G+K],ABW}
 for the finite dimensional case to construct a sequence of  solutions of approximated equations and obtain a prior estimates for those approximated solutions.
Then we show that the limit of those approximated solutions solves
the original equation by using the local monotonicity arguments.

~~

Now let us describe the framework in more detail.  Let
$$V\subset H\equiv H^*\subset V^*$$
 be a Gelfand triple, $i.e.$  $(H, \<\cdot,\cdot\>_H)$ is a separable
Hilbert space which is  identified with its dual space by the Riesz Lemma, $V$ is a reflexive  Banach space  that  is
 continuously and densely embedded into $H$. If  ${ }_{V^*}\<\cdot,\cdot\>_V$ denotes  the duality
between  $V$ and its dual space $V^*$,  then we have
$$ { }_{V^*}\<u, v\>_V=\<u, v\>_H, \  u\in H ,v\in V.$$
Let $(\Omega,\mathbb{P},\mathbb{F}, \mathcal{F})$, where $\mathbb{F}=(\mathcal{F}_t)_{t\geq0}$,   is a filtered probability space,
  $(Z,\mathcal{Z})$ be a measurable space, and $\nu$ be a $\sigma$-finite measure on it. We write
  $$\tilde{N}((0,t]\times B)=N((0,t]\times B)-t\nu(B),\ t\geq0,\ B\in\mathcal{Z} $$
  for the compensated Poisson random measure on $[0,T]\times\Omega\times Z$ associated with a stationary Poisson point process $p$ (see Section 2 for more details).
A typical example of  $N$  is  a Poisson random measure associated with a L\'{e}vy process
taking values in a separable Banach space.
Let $U$ be a separable Hilbert space and let  us denote by  $\left(\mathcal{T}_2(U;H),
 \|\cdot\|_2\right) $
the Hilbert space of all Hilbert-Schmidt operators from $U$ to $H$. Assume that    $\{W_t\}_{t\geq0}$ is a $U$-valued  cylindrical Wiener process
on the probability space $(\Omega,\mathbb{P},\mathbb{F}, \mathcal{F})$.
We use the symbol $\mathcal{P}$ to denote the predictable $\sigma$-field, i.e. the $\sigma$-field generated by all left continuous and $\mathbb{F}$-adapted
 real-valued processes on $[0,T]\times\Omega$.
We shall denote by $\mathcal{B}\mathcal{F}$ the $\sigma$-field of the progressively measurable sets on $[0,T]\times\Omega$, i.e.
      $$\mathcal{B}\mathcal{F}= \{ A\subset[0,T]\times\Omega:\forall\ t\in[0,T], A\cap([0,t]\times\Omega)\in\mathcal{B}([0,t])\otimes\mathcal{F}_t\}.$$
Now  we
consider a type of SPDEs driven by L\'{e}vy processes of the following form:
 \begin{align}
 \begin{split}\label{SEE}
\d X_t&=A(t,X_t) \d t+B(t,X_t) \d W_t\\
&~~~~+\int_{D^c}f(t,X_{t-},z)\tilde{N}(\d t,\d z)+\int_Dg(t,X_{t-},z)N(\d t, \d z),\\
 X_0&=x,
 \end{split}
\end{align}
 where $x$ is an $\mathcal{F}_0$-measurable random variable, $A: [0,T]\times \Omega\times V\to V^* $ and $B:
[0,T]\times  \Omega\times V\to \mathcal{T}_2(U;H)$ are both $\mathcal{B}\mathcal{F}\otimes\mathcal{B}(V)$-measurable functions,
$D\in\mathcal{Z}$ with $\E N((0,t]\times D)<\infty$  for every $0<t\leq T$,   and  $f,g:[0,T]\times\Omega\times V\times Z\to H
$ are $\mathcal{P}\otimes\mathcal{B}(V)\otimes\mathcal{Z}$-measurable functions.

The main aim of this work is to establish the existence and uniqueness of strong solutions to (\ref{SEE}) under the coercivity and  local monotonicity conditions.

For this purpose, let us first formulate the main  assumptions on the coefficients.

~~

Suppose that there exist constants
  $\alpha>1$, $\beta\ge 0$, $\theta>0$, $C>0$,  a positive $\mathbb{F}$-adapted process $F$ and  a measurable,  bounded on balls function $\rho: V\rightarrow [0,+\infty)$
 such that the
 following
 conditions hold for all $v,v_1,v_2\in V$ and $(t,\omega)\in [0,T]\times \Omega$:
\begin{enumerate}
    \item [$(H1)$] (Hemicontinuity) The map
     $ s\mapsto { }_{V^*}\<A(t,v_1+s
 v_2),v\>_V$
  is  continuous on $\mathbb{R}$.

    \item [$(H2)$] (Local monotonicity)
\begin{align*}
& 2 { }_{V^*}\<A(t,v_1)-A(t,v_2), v_1-v_2\>_V
    +\|B(t,v_1)-B(t,v_2)\|_{2}^2\\
  + & \int_{D^c}\|f(t,v_1,z)-f(t,v_2,z)\|^2_H\nu(\d z)
\le \left(C + \rho(v_2) \right)\|v_1-v_2\|_H^2,
     \end{align*}
\item [$(H3)$] (Coercivity)
    \begin{align*}
     2 { }_{V^*}\<A(t,v), v\>_V +\|B(t,v)\|_{2}^2
    +\theta
    \|v\|_V^{\alpha} \le F_t + C\|v\|_H^2.
    \end{align*}

\item[$(H4)$] (Growth)
$$ \|A(t,v)\|_{V^*}^{\frac{\alpha}{\alpha-1}} \le \left(F_t +
 C\|v\|_V^{\alpha}\right) \left( 1 +\|v\|_H^{\beta} \right).$$
    \end{enumerate}

\beg{defn} (Solution of SEE) An  $H$-valued  c\`{a}dl\`{a}g
$\mathbb{F}$-adapted process $\{X_t\}_{t\in [0,T]}$ is called a
solution of $(\ref{SEE})$, if for its $\d t\times \P$-equivalent
class $\bar{X}$ we have
\begin{enumerate}
\item[(1)] $\bar{X}\in L^\alpha([0,T];  V)\cap L^2([0,T];  H)$,  $\P$-a.s.;
\item[(2)] the following equality holds $\P$-a.s.:
\begin{align*}
X_t=x&+\int_0^t A(s, \bar{X}_s)\d s+\int_0^t B(s, \bar{X}_s)\d W_s\\
&+\int_0^t\int_{D^c}f(s,\bar{X}_{s-},z)\tilde{N}(\d s,\d z)
+\int_0^t\int_{D}g(s,\bar{X}_{s-},z)N(\d s,\d z),\ t\in[0,T].
\end{align*}
\end{enumerate}
\end{defn}
\begin{rem}
The integrability of all terms in the  above equality  are implicitly required in the definition  and it will be all  justified in the  proof of existence of solutions.
Note that $A(s, \bar{X}_s)$ is a $V^*$-valued process according to the definition, however,
 the  integral with respect to $\d s$ in the above equality is initially a $V^*$-valued Bochner integral which turns out to be in fact $H$-valued.
\end{rem}

Now we can  present the main result of this  paper.

\beg{thm}\label{T1}
 Suppose that conditions $(H1)-(H4)$ hold for  $F\in L^{\frac{\beta+2}{2}}([0,T]\times \Omega; \d
    t\times \mathbb{P})$, and
 there exists  constant  $\gamma<\frac{\theta}{2\beta}$   such that for all $t\in[0,T], \omega\in \Omega$  and  $v\in V$ we have
\begin{align}
& \|B(t,v)\|_2^2  + \int_{D^c}\|f(t,v,z)\|_H^2\nu(\d z) \le   F_t+C\|v\|_H^2+ \gamma \|v\|^{\alpha}_V; \label{c3}\\
& \int_{D^c} \|f(t,v,z)\|^{\beta+2}_H\nu(\d z)\leq   F_t^{\frac{\beta+2}{2}}+ C\|v\|^{\beta+2}_H ;  \label{c4}\\
 & \rho(v) \le C(1+\|v\|_V^\alpha) (1+\|v\|_H^\beta).\label{c5}
\end{align}

\noindent  (i) Then  for any  $x \in L^{\beta+2}(\Omega, \mathcal{F}_0,\mathbb{P};H)$,
   Equation $(\ref{SEE})$
    has a unique solution $\{X_t\}_{t\in [0,T]} $.

\noindent (ii) If  $g\equiv 0$, then there exists a constant $C$ such that
\begin{equation}\label{e0}
  \sup_{t\in[0,T]} \E \|X_t\|_H^{\beta+2}  + \E \int_0^T \|X_t\|_H^\beta \|X_t\|_V^\alpha \d t
 \le C\left(\E\|x \|_H^{\beta+2}+ \E\int_0^T F_t^{(\beta+2)/2}\d t\right).
\end{equation}

\noindent (iii) If $g\equiv 0$ and $\gamma$ is small enough, then we have
\begin{equation}\label{e1}
 \E\left(\sup_{t\in[0,T]}\|X_t\|_H^{\beta+2}\right)  + \E \int_0^T \|X_t\|_H^\beta \|X_t\|_V^\alpha \d t
    \le C\left(\E\|x \|_H^{\beta+2}+ \E\int_0^T F_t^{(\beta+2)/2}\d t\right).
\end{equation}

\end{thm}
\begin{rem} (1) If $f=g\equiv0$ in  (\ref{SEE}) (i.e.   Wiener noise case),  then Theorem \ref{T1} recovers
the main result in \cite{Liu+Roc}. Moreover, we improve
  \cite[Theorem 1.1]{Liu+Roc}   for allowing a  positive constant $\gamma$ in (\ref{c3}), which means that the diffusion coefficient $B$
can also depend on some gradient term of the solution in
applications.
 We also  want to emphasize that
 $(H2)$ is essentially weaker than the classical  monotonicity condition used extensively  in the literature  (i.e. $\rho\equiv0$, see e.g. \cite{Par75,KR79,PR07,RZ07,RRW,GLR}).
The  typical examples are  the stochastic Burgers equations and
2D   Navier-Stokes equation (see Remark \ref{rem4.1} for many other examples) on
a bounded or unbounded domain, which satisfies $(H2)$ but does not
satisfy the standard monotonicity condition (cf. Section 4  for the details).

(2) If $g\equiv0$ in (\ref{SEE}), $\rho\equiv 0, \alpha=2, \beta=0$ in $(H2)$-$(H4)$,  then the existence and uniqueness of strong solutions to (\ref{SEE}).
follows from the general result of Gy\"{o}ngy \cite{G}.

(3) If the noise is zero or additive type  in (\ref{SEE}), then the (local)
existence and uniqueness of solutions  is  established in
\cite{L,LR13} by replacing $(H2)$ with the following more general
local monotonicity condition:
$$ { }_{V^*}\<A(t,v_1)-A(t,v_2), v_1-v_2\>_V
     \le \left(K +\eta(v_1)+ \rho(v_2) \right)\|v_1-v_2\|_H^2,$$
where $\eta,\rho: V\rightarrow [0,+\infty)$ are measurable functions
and locally bounded in $V$.

(4) In general, the estimates (\ref{e0}) and (\ref{e1}) might not
hold anymore if we have  large jumps term in the  equation. However,
if we assume that the L\'{e}vy measure has finite moment of certain
order (see e.g.\cite{DXZ}), then it is still possible to obtain some
similar estimates. This subject and some related applications will be investigated in future works.
\end{rem}

\begin{rem}
(1) Note that if $\beta=0$ in $(H4)$, then one can just take any  $\gamma<\infty$ in \eqref{c3}.
In this case, the assumption on $B$ in  (\ref{c3}) can be removed since it  follows directly from $(H3)$ and $(H4)$ (cf. \cite[Remark 4.1.1]{PR07}).

(2)  If $f$ satisfies  the following   growth condition for some fixed $p\ge \beta+2$:
    \begin{align*}
            &\|f(t,v,z)\|^p_H\leq h(z)^p(F_t^{\frac{p}{2}}+C\|v\|^p_H ),\
\  (t,v,z)\in [0,T] \times V\times D^c,
    \end{align*}
where $\int_{D^c} \left[h(z)^{\beta+2} + h(z)^2\right] \nu(dz) <\infty$,
 then it is easy to show that  conditions \eqref{c4}  and (\ref{c3})  hold.

In particular,  if $f$ satisfying the following conditions:
\begin{align*}
 &  \|f(t,x,z)-f(t, y,z)\|_H \le C \|x-y\|_H \|z\|, \ t\in[0,T], \  x,y\in V, \ z\in D^c;       \\
  &  \|f(t,x,z)\|_H  \le C (1+\|x\|_H) \|z\|, \   t\in[0,T],\  x\in V, \ z\in D^c,
\end{align*}
where $\int_{D^c}\|z\|^2\nu(\d z)<\infty$, then $(H2)$, \eqref{c3}
and (\ref{c4}) are all  fulfilled.
\end{rem}

The rest of the paper is organized as follows:  in the next section we will recall some
preliminaries on the Poisson random measure and its corresponding stochastic integral.
 The proof of the main result  will be given in Section 3 and some concrete examples of SPDE will
be studied in  Section 4 as  applications. Note that we always use $C$ to denote a generic constant
which may change from line to line.

\section{Some Preliminaries on Poisson  Random Measure}

We begin with a brief review of  terminology and results on Poisson random measures.
Let $(S,\mathcal{S})$ be a measurable space, $\mathbb{N}=\{0,1,2,\cdots\}$ and $\bar{\mathbb{N}}=\mathbb{N}\cup\{\infty\}$.
Let $\mathbb{M}_{\bar{\mathbb{N}}}(S)$ denote the space of all $\bar{\mathbb{N}}$-valued measures on $(S,\mathcal{S})$.
 We use the symbol $\mathcal{B}(\mathbb{M}_{\bar{\mathbb{N}}}(S))$ to denote the smallest
$\sigma$-field on $\mathbb{M}_{\bar{\mathbb{N}}}(S)$ with respect to which all  mappings
$i_B:\mathbb{M}_{\bar{\mathbb{N}}}(S)\ni\mu\mapsto\mu(B)\in\bar{\mathbb{N}}$, $B\in\mathcal{S}$ are measurable.

\begin{defn} A map $N:\Omega\times \mathcal{S}\rightarrow \bar{\mathbb{N}}$ is called an $\bar{\mathbb{N}}$-valued
\textbf{random measure} if for each $\omega\in\Omega$, $N(\omega,\cdot)\in \mathbb{M}_{\bar{\mathbb{N}}}(S)$  and
for each $A\in\mathcal{S}$, $N(\cdot,A)$ is an $\bar{\N}$-valued random variable on the probability
space $(\Omega,\mathbb{P},\mathcal{F})$. We will often write $N(A)$ instead of $N(\cdot,A)$ for simplicity of notation.
\end{defn}

\begin{defn}\label{De-Poisson}
  An $\bar{\mathbb{N}}$-valued random measure $N$  is called a \textbf{Poisson random measure} if
  \begin{enumerate}
  \item[(1)]for any $B\in\mathcal{S}$ satisfying $\mathbb{E}[ N(B)]<\infty$, $N(B)$ is a Poisson random variable  with  parameter $\eta(B)=\E [N(B) ]$;
  \item[(2)]for any pairwise disjoint sets $B_1,\cdots,B_n\in\mathcal{S}$, the random variables
         $$N(B_1),\ \cdots,\ N(B_n)$$
         are independent.
  \end{enumerate}
\end{defn}

Let $(Z,\mathcal{Z})$ be a measurable space. A \textbf{point function} $\alpha$ on $(Z,\mathcal{Z})$ is a
mapping $\alpha:\mathcal{D}(\alpha)\rightarrow Z$, where the domain
$\mathcal{D}(\alpha)$ of $\alpha$ is a countable subset of $(0,\infty)$. Let $\Pi_{Z}$ be the set of all point
 functions on $Z$.   For each point function, we define a counting measure $N_{\alpha}$ by
 \begin{align*}
          N_{\alpha}(U):=\sharp\{s\in  \mathcal{D}(\alpha) :\ (s,\alpha(s))\in U\},\ \ U\in\mathcal{B}((0,\infty))\otimes\mathcal{Z} .
 \end{align*}
 Denote by $\mathcal{Q}$ the $\sigma$-field on $\Pi_{Z}$
 generated by all the subsets $\{\alpha\in\Pi_Z:\ N_{\alpha}(U)=k\}$, $U\in \mathcal{Z}$, $k=0,1,2,\cdots$.
 A function $p:\Omega\rightarrow\Pi_{Z}$ is called a \textbf{point process} on $Z$ if it is $\mathcal{F}/\mathcal{Q}$-measurable.
Let $p$ be a point process on  $Z$. We can define  the counting measure $N_p$ associated with $p$ by
  \begin{align}\label{eq-N-0}
     N_{p}(U,\omega):=\sharp\{s\in\mathcal{D}(p(\omega)):(s,p(s,\omega))\in U\},\ \ U\in\mathcal{B}((0,\infty))\otimes\mathcal{Z}, \  \omega\in\Omega.
 \end{align}
 In particular, we have
 \begin{align}\label{eq-N-1}
     N_{p}((0,t]\times A,\omega)=\sharp\{s\in(0,t]\cap\mathcal{D}(p(\omega)):p(s,\omega)\in A\},\ \ A\in\mathcal{Z},\ \  0<t\leq T.
 \end{align}

It is also useful to introduce the shifted point process $\theta_tp, \ t\ge 0$ defined by
 \begin{align*}
 (\theta_tp)(s)&=p(s+t),\ s>0;\\
  \mathcal{D}(\theta_t p)&=\{s\in(0,\infty):s+t\in\mathcal{D}(p)\}.
  \end{align*}
and the stopped point process $\alpha_tp$ defined  by
  \begin{align*}
       (\alpha_tp)(s)&=p(s),\ \text{for }s\in\mathcal{D}(\alpha_tp);\\
       \mathcal{D}(\alpha_tp)&=(0,t]\cap\mathcal{D}(p).
  \end{align*}
  \begin{defn}
   A point process $p$ is said to be \textbf{finite} if $\E N_{p}((0,t]\times Z)<\infty$  for every $0<t\leq T$.

 A point process $p$ is said to be \textbf{$\sigma$-finite} if there exists
an increasing sequence $\{D_n\}_{n\in\mathbb{N}}\subset\mathcal{Z}$ such that $\cup_{n}D_n=Z$ and $\E
N_{p}((0,t]\times D_n)<\infty$ for all $0<t \leq T$ and $n\in\mathbb{N}$.

 A point process $p$ is said to be \textbf{stationary} if for every $t>0$, $p$ and $\theta_tp$ have the same probability laws.

  A point process $p$ is said to be \textbf{renewal} if  it is stationary and for every $0<t<\infty$, the point processes $\alpha_tp$ and $\theta_tp$ are independent.

  A point process $p$ is said to be \textbf{adapted} to the filtration $\mathbb{F}$ if for every $t>0$ and $A\in\mathcal{Z}$, its counting measure $N_{p}((0,t]\times A)$ is $\mathcal{F}_t$-measurable.

 A point process $p$ is called a \textbf{Poisson point process} if $N_{p}(\cdot)$ defined by \eqref{eq-N-0} is a Poisson random measure on $((0,\infty)\times Z,\mathcal{B}((0,\infty))\otimes\mathcal{Z})$.
 \end{defn}
 \begin{rem}\label{rem-1}
 It can be shown that if a point process $p$ is $\sigma$-finite and renewal, then $N_{p}$ defined by \eqref{eq-N-0}
is a Poisson random measure (cf. \cite[Theorem 3.1]{[Ito-0]}).  It is easy to verify that a
Poisson point process is stationary if and only if there exists a
nonnegative  measure $\nu$ on $(Z,\mathcal{Z})$ such that
 \begin{align}\label{eq-102}
         \E N_{p}((0,t]\times A)=t\nu(A),\ \ \ t> 0,\ \ A\in\mathcal{Z}.
 \end{align}
 In such a case, we say that the Poisson random measure $N_{p}$ is time homogenous. At this point,
 it should be mentioned that, in the literature,  some authors may use the above property \eqref{eq-102} as an alternative definition
 of stationary property of a Poisson point process. \dela{Actually,} In fact,  this is consistent with our  definition of a stationary point process.
\end{rem}

  Let $\mathcal{M}_T^q(\mathcal{P}\otimes\mathcal{Z},\d t\times\mathbb{P}\times\nu;H)$, $q\in [1,\infty)$,
\dela{(resp. $\mathcal{M}_T^1(\mathcal{P}\otimes\mathcal{Z},\d t\times\mathbb{P}\times\nu;H)$)}
be the space of all  (equivalence classes of) $\mathcal{P}\otimes\mathcal{Z}$-measurable functions
$f:[0,T]\times\Omega\times Z\to H$ such that
\begin{align}\label{eq-m-loc}
     \E \int_0^T\int_{Z} \|f(s,\cdot,z)\|_H^q \nu(\d z)\d s<\infty\dela{\
\left( \text{resp.}\ \E \int_0^T\int_{Z} \|f(s,\cdot,z)\|_H \nu(\d z)\d s<\infty  \right)}.
\end{align}
Let  $\mathcal{M}_T(\mathcal{P}\otimes\mathcal{Z},N;H)$ be the space of all $\mathcal{P}\otimes\mathcal{Z}$-measurable functions $f:[0,T]\times\Omega\times Z\to H$ such that
\begin{align}\label{eq-m-loc-1}
     \E \int_0^T\int_Z \|f(s,\cdot,z)\|_H N(\d s,\d z)<\infty.
\end{align}
Here  $\int_0^T\int_Z\|f(s,\omega,z)\|_HN(\d s, \d z)(\omega)$ is understood to be the Lebesgue integral $w.r.t.$
the measure $N(\cdot,\cdot)(\omega)$ for every $\omega\in\Omega$ and  is equal to the convergent sum (cf. \cite{[Ikeda]}),
\begin{align*}
\int_0^T\int_Z \|f(s,\omega, z)\|_H N(\d s, \d z)(\omega) =
\sum_{s\in(0,T]\cap\mathcal{D}(p(\omega))}\|f(s,\omega,
p(s,\omega))\|_H.
\end{align*}
It should come as no surprise that if $f:[0,T]\times\Omega\times Z\rightarrow H$ is a $\mathcal{B}([0,T])\otimes\mathcal{F}_T\otimes\mathcal{Z}$-measurable function
and $\E\int_0^T\int_Z\|f(s,\cdot,z)\|_HN(ds,dz)<\infty,$ then for every $\omega\in\Omega$, $f(\cdot,\omega,\cdot)$ is $\mathcal{B}([0,T])\otimes\mathcal{Z}$-measurable
and $\int_0^T\int_Z\|f(s,\omega,z)\|_HN(\d s,\d z)(\omega)<\infty$, $\mathbb{P}$-a.s..  Hence for almost all $\omega\in\Omega$, $f(\cdot,\omega,\cdot)$ is Bochner integrable with respect to $N(\d s,\d z)(\omega)$ and we have for every $t\leq T$
\begin{align}\label{sec-2-eq-30}
    \int_0^t\int_Z f(s,\omega,z) N(\d s,\d z)(\omega)=\sum_{s\in(0,t]\cap\mathcal{D}(p(\omega))}f(s,\omega,p(s,\omega)),\ \mathbb{P}\text{-a.s.}
\end{align}

Now  we state  some important properties of the stochastic integrals w.r.t. the compensated Poisson random
measures, where the proofs of these properties and more detailed discussions can be found in
 \cite{[Ikeda]} (see also \cite{[Brz-Hau], [Rud-int],[Zhu]}).

\begin{prp}\label{P2.1}
 Assume $f\in \mathcal{M}_T^2(\mathcal{P}\otimes\mathcal{Z},\d t\times\mathbb{P}\times\nu;H)$. Then the following conclusions hold:
\begin{enumerate}
\item[(i)]
The stochastic integral process $\int_0^t\int_Z f(s,\cdot,z)\tilde{N}(\d s,\d z)$, $t\in[0,T]$ is a
c\`{a}dl\`{a}g $2$-integrable martingale. More precisely, it has a modification which has c\`{a}dl\`{a}g trajectories.

\item[(ii)] The  isometry property:
\begin{align}\label{eq-iso}
\mathbb{E}\big{\|}\int_0^t\int_Z f(s,\cdot,z)\tilde{N}(\d s, \d z)\big{\|}_H^2
=\mathbb{E}\int_0^t\int_Z\|f(s,\cdot,z)\|_H^2\nu(\d z)\d s,\ \ t\in(0,T].
\end{align}

\item[(iii)] If $D\in\mathcal{Z}$ with $\mathbb{E}(N((0,t]\times  D))<\infty$, then for every
$t\in[0,T]$, $\mathbb{P}\text{-a.s.}$,
\begin{align}\label{prop_eq_1}
     \int_0^t\int_Df(s,\cdot,z)\tilde{N}(\d s,\d z)
          =\sum_{s\in(0,t]\cap
     \mathcal{D}(p)}f(s,\cdot,p(s))1_{D}(p(s))-\int_0^t\int_Df(s,\cdot,z)\nu(\d z)\d s ;
\end{align}

\item[(iv)] If in addition
$f\in
 \mathcal{M}_T^1(\mathcal{P}\otimes\mathcal{Z},\d t\times\mathbb{P}\times\nu;H)$, then for each $t\in[0,T]$, $\mathbb{P}$-a.s.
\begin{align}\label{eq-108}
     \int_0^t\int_Zf(s,\cdot,z)\tilde{N}(\d s,\d z)
    = \sum_{s\in(0,t]\cap
     \mathcal{D}(p)}f(s,\cdot,p(s))-\int_0^t\int_Zf(s,\cdot,z)\nu(\d z)\d s.
\end{align}
\end{enumerate}
\end{prp}

\begin{rem}
\begin{enumerate}
    \item[(1)] We may extend the stochastic integral to $\mathcal{P}\otimes\mathcal{Z}$-measurable functions $f$ satisfying
$$   \int_0^T\int_{Z} \|f(s, \cdot, z)\|_H^2 \nu(\d z)\d s<\infty,\  \P\text{-a.s.}. $$
In this case, the  process $\int_0^t\int_Z f(s,\cdot,z)\tilde{N}(\d s,\d z)$, $t\in[0,T]$ is a
c\`{a}dl\`{a}g $2$-integrable local martingale and for every stopping time $\tau\leq T$, we have
            \begin{align*}
    \int_0^{t\wedge\tau}\int_Zf(s,\cdot,z)\tilde{N}(\d s,\d z)=\int_0^t\int_Z1_{[0,\tau]}f(s,\cdot,z)\tilde{N}(\d s,\d z).
    \end{align*}
\item[(2)] From now on, whenever we use the stochastic  process
$\int_0^t\int_Z f(s,\cdot,z)\tilde{N}(\d s,\d z)$, $t\in[0,T]$, we implicitly assume that it
has c\`{a}dl\`{a}g trajectories, in which case, the stochastically equivalence coincides with the $\mathbb{P}$-equivalence.

\item[(3)] For  Banach spaces martingale type $p$ ($1<p\leq 2$), one has, instead of the It\^{o}  isometry property \eqref{eq-iso},
 the following continuity property ($C_p$ is some constant)(cf. \cite{[Zhu]}):
\begin{align*}
\mathbb{E}\big{\|}\int_0^T\int_Z f(s,\cdot,z)\tilde{N}(\d s,\d z)\big{\|}_H^p
\leq C_p \mathbb{E}\int_0^T\int_Z\|f(s,\cdot,z)\|_H^p\nu(\d z)\d s.
\end{align*}
\item[(4)] Even though there are close connections between predictable processes and progressively measurable processes,
 the  predictability requirement of the function $f$ in Proposition \ref{P2.1} $(iii)$ and $(iv)$ is necessary.
 In fact, one can find a progressively measurable but not predictable function such that identities
\eqref{prop_eq_1} and \eqref{eq-108} no longer hold (cf. \cite{[Zhu]}).
\end{enumerate}
\end{rem}

One should note that  another important and widely   used class of Poisson random measures
are the  one associated to a L\'{e}vy process,  which is actually  a special type of Poisson random measures associated to a Poisson point process as we  discussed before.
More precisely,  let $L:=(L_t)_{t\geq0}$ be an  $Z$-valued L\'{e}vy process, where $Z$ is a separable Banach space.
Without loss of generality we may  assume
 that the L\'{e}vy process $L$ is c\`{a}dl\`{a}g, even if we don't impose the c\`{a}dl\`{a}g property in the definition of a L\'{e}vy process,
see e.g.   \cite[Theorem 16.1]{[Dinculeanu]}. Hence,
for every $\omega\in\Omega$, $L_\cdot(\omega)$ has at most  countable number of jumps on $[0,t]$.
So it is easy to see that for every $\omega\in\Omega$,
 $$\triangle L_{\cdot}(\omega): [0, \infty)\rightarrow Z; \    \triangle L_s(\omega):=L_s(\omega)- L_{s-}(\omega)$$
 is a point function in $(Z\setminus\{0\},\mathcal{B}(Z\setminus\{0\}))$.
   Let us define a function $N$ by
  \begin{align}\label{eq-N-4}
     N(A,\omega)=\sharp\{s\in(0,\infty):(s,\triangle L_s(\omega))\in A\},\ \ A\in\mathcal{B}((0,\infty))\otimes\mathcal{B}( Z\setminus\{0\}),\ \omega\in\Omega.
 \end{align}
It is easy to check that
 $\triangle L:\Omega\rightarrow \Pi_Z$ is $\mathcal{F}/\mathcal{Q}$-measurable. Thereby  $\triangle L$ is a point process.
Since the L\'{e}vy process $L$ has independent and stationary increments, one can show  that the point process $\triangle L$ is stationary and renewal.  Obviously,  by taking
$D_n=\{x\in Z: \|x\|>\frac{1}{n}\}$, we find that the point process $\triangle L$ is $\sigma$-finite.
On the basis of Remark \ref{rem-1}, we know that the function $N$ defined by equality \eqref{eq-N-4} is a
stationary Poisson random measure  $$\E N((0,t]\times A)=t\nu(A),\ t>0,\ A\in\mathcal{B}(Z\setminus\{0\}),$$
where $\nu$ is a nonnegative measure.
In this case,   $N$ is called the Poisson random measure associated to the L\'{e}vy process $L$.

\section{Proof of The  Theorem \ref{T1}}

\subsection{The case without large jumps}

First of all we note that since $\nu(D)<\infty$,  the set
$$\{s\in(0,T]\cap\mathcal{D}(p):p(s,\omega)\in D\}$$
 contains only finitely many points for almost all $\omega\in\Omega$.
Put
\begin{align*}
\tau_1&=\inf\{s\in(0,\infty)\cap\mathcal{D}(p):p(s)\in D\}\wedge T;\\
  \tau_m&=\inf\{s\in(0,\infty)\cap\mathcal{D}(p): p(s)\in D; s>\tau_{m-1}\}\wedge T,\ m\geq2.
\end{align*}
        The random times $\tau_1,\tau_2,\cdots$ form a random configuration of points in $(0,T]$ with $p(\tau_i)\in D$ and it is a sequence
        of jump times of the Poisson process $N(t,D):=N((0,t]\times D)$, $t\in(0,T]$. We see at  once that $\tau_m\uparrow T$ as $m\rightarrow\infty$ $\mathbb{P}$-a.s. and
 for each $m$, the random time $\tau_m$ is a stopping time. 
 Note that since $\int_{0}^tg(s,X_{s-}, z)N(\d s, \d z)=0$  for $t\in[0,\tau_{1})$,
 the equation \eqref{SEE} on the interval $[0,\tau_1)$ can be rewritten into the following type of equation:
\begin{align}\begin{split}\label{SEE-1}
    & \d X_t=A(t,X_t)\d t+B(t,X_t)\d W_t+\int_{D^c}f(t,X_{t-},z)\tilde{N}(\d t,\d z),  \ t\in [0,\tau_1), \\
     &X_0=x.
\end{split}
\end{align}
Actually, by means of the interlacing procedure (which will be introduced in Section 3.2 and cf. also Theorem 9.1 in \cite{[Ikeda]}),
 for the proof of Theorem \ref{T1}
  it is sufficient to  show  the existence and uniqueness of solutions to (\ref{SEE-1}).

    \beg{thm}\label{T2}
    Under the  assumptions of Theorem \ref{T1}, for every
 $x  \in L^{\beta+2}(\Omega, \mathcal{F}_0,\mathbb{P};H)$, there exists a unique c\`{a}dl\`{a}g $H$-valued $\mathbb{F}$-adapted process $(X_t)$ such that
 $\mathbb{P}$-a.s.:
              \begin{align}
                          X_t=x +\int_0^tA(s,\bar{X}_s)\d s+\int_0^t B(s,\bar{X}_s)\d W_s+\int_0^t\int_{D^c} f(s,\bar{X}_{s-},z)\tilde{N}(\d s,\d z),\ t\in[0,T],
              \end{align}
        where $\bar{X}\in L^\alpha([0,T]\times \Omega, \d t\times\P; V)\cap L^2([0,T]\times \Omega, \d t\times\P; H)$ and
 it is $\d t\times\mathbb{P}$-equivalent to $X$.

   Moreover, we have
\begin{equation}
  \sup_{t\in[0,T]} \E \|X_t\|_H^{\beta+2}  + \E \int_0^T \|X_t\|_H^\beta \|X_t\|_V^\alpha \d t
 \le C\left(\E\|x \|_H^{\beta+2}+ \E\int_0^T F_t^{(\beta+2)/2}\d t\right).
\end{equation}
     \end{thm}


 The proof of Theorem \ref{T2} is divided into three steps.
Assume that
$\{e_1,e_2,\cdots \}\subset V$ is an orthonormal basis of $H$ such that $span\{e_1,e_2,\cdots\}$ is dense in $V$.
Denote $H_n:=span\{e_1,\cdots,e_n\}$. Let $P_n:V^*\rightarrow H_n$ be defined by
$$ P_ny:=\sum_{i=1}^n { }_{V^*}\<y,e_i\>_V e_i, \ y\in V^*.  $$
It is easy to see that $P_n|_H$ is just  the orthogonal projection onto $H_n$ in $H$ and we have
$$ { }_{V^*}\<P_nA(t,u), v\>_V=\<P_nA(t,u),v\>_H={ }_{V^*}\<A(t,u),v\>_V, \ u\in V, v\in H_n.  $$
Let $\{g_1,g_2,\cdots \}$ be an orthonormal basis of $U$ and
$$ W^{(n)}_t:=\sum_{i=1}^n\<W_t,g_i\>_U g_i=\tilde{P}_n W_t, $$
where
 $\tilde{P}_n$ is the orthogonal projection from $U$ onto $span\{g_1,\cdots,g_n\}$.

For each
$n\in \mathbb{N}$, we consider the following stochastic differential
equation on $H_n$:
\begin{align}
\begin{split} \label{approximation}
&\d X_t^{(n)}=P_n A(t,X_t^{(n)}) \d t+ P_n B(t,X_t^{(n)}) \d
W_t^{(n)}+\int_{D^c}P_nf(t,X_{t-}^{(n)},z)\tilde{N}(\d t,\d z), \\
&\ X_0^{(n)}=P_n x.
\end{split}
\end{align}

According to \cite[Theorem 1]{[G+K]} (cf. also  \cite[Theorem 3.1]{ABW}),
  \eqref{approximation} has a unique  strong solution, i.e. satisfying the following integral equation:
\begin{align}\label{eq-solution}
        X_t^{(n)}=&P_n x+\int_{0}^tP_n A(s,X_s^{(n)}) \d t+\int_{0}^t P_n B(s,X_s^{(n)}) \d
W_s^{(n)}\\
&+\int_{0}^t\int_{D^c}P_nf(s,X_{s-}^{(n)},z)\tilde{N}(\d s,\d z), \ \ t\in[0,T]. \nonumber
\end{align}

In order to construct the solution to equation (\ref{SEE-1}), we need to find a
priori estimate for $X^{(n)}$.

\begin{lem}\label{L2}
 Under the  assumptions of Theorem \ref{T1}, there exists $C>0$ such that
\begin{equation}\begin{split}\label{l2'}
&  \sup_{n\in\mathbb{N}}\Big{(}\sup_{t\in[0,T]} \E\|X_t^{(n)}\|_H^{\beta+2} +\E\int_{0}^T \|X_t^{(n)}\|_H^{\beta}\|X_t^{(n)}\|_V^\alpha \d t\Big{)} \\
 \le& C\left(\E\|x \|_H^{\beta+2}+ \E\int_0^T F_t^{(\beta+2)/2}\d t\right).
\end{split}
\end{equation}
\end{lem}

\begin{proof} For any given $n\in\mathbb{N}$, we define
\begin{align}\label{stopping time}
      \tau_R^{(n)}:=\inf\{t\geq0:\|X_t^{(n)}\|_H>R\}\wedge T.
\end{align}
Since the solution $(X_t^{(n)})_{0\leq t\leq T}$ is right continuous and $\mathbb{F}$-adapted,
 $\tau_R^{(n)}$ is a stopping time   for every $R\in\mathbb{N}$.
Moreover, since the trajectories of $X^{(n)}$ are $\mathbb{P}$-a.s. right  continuous with left limits ,
the process $X^{(n)}$ is bounded on every compact intervals.  Hence we see that $\tau_R^{(n)}\uparrow T$
, $\mathbb{P}$-a.s. and $\mathbb{P}\{\tau_R^{(n)}<T\}=0$ as $R\rightarrow\infty$.

 For the simplicity of notations we take $p=\beta+2$. By applying  the It\^{o} formula (cf. \cite{[Met]}) to the function $\|\cdot\|_H^p$ and the process $X_{t}^{(n)}$  we have
\begin{align}\begin{split}\label{Ito estimate 2}
\|X_{t}^{(n)}\|_H^p=&\|X_0^{(n)}\|_H^p
+p(p-2)\int_{0}^{t}\|X_{s-}^{(n)}\|_H^{p-4}\|
(P_nB(s,X_s^{(n)})\tilde{P}_n)^* X_{s-}^{(n)}\|_{H}^2 \d s\\
&+\frac{p}{2}\int_{0}^{t}\|X_{s-}^{(n)}\|_H^{p-2}\left( 2 {
}_{V^*}\<A(s,X_s^{(n)}),X_{s-}^{(n)}\>_V+\|P_nB(s,X_s^{(n)})\tilde{P}_n\|_{2}^2\right)
\d s
\\
&+\int_{0}^{t}p\|X_{s-}^{(n)}\|_H^{p-2}\<X_{s-}^{(n)},P_n B(s,X_s^{(n)})\d
W_s^{(n)}\>_H\\
&+\int_{0}^{t}\int_{D^c}p\|X_{s-}^{(n)}\|^{p-2}_H\<X_{s-}^{(n)},P_nf(s,X_{s-}^{(n)},z)\>_H\tilde{N}(\d s,\d z)\\
&+\int_{0}^{t}\int_{D^c}\Big{[}\|X_{s-}^{(n)}+P_nf(s,X_{s-}^{(n)},z)\|^p_H-\|X_{s-}^{(n)}\|^p_H\\
&\hspace{1cm}-p\|X_{s-}^{(n)}\|_H^{p-2}\<X_{s-}^{(n)},P_nf(s,X_{s-}^{(n)},z)\>_H\Big{]}N(\d s,\d z),
\  \ t\in[0,T].
\end{split}
\end{align}
Then assumption $(H3)$ and \eqref{c3}  imply that
\begin{align*}
      &    \|X_t^{(n)}\|_H^p + \frac{p\theta}{2} \int_{0}^{t}  \|X_s^{(n)}\|_H^{p-2}\|X_s^{(n)}\|_V^{\alpha}\d s\\
 \le & \|x\|_H^p+p(p-2)\int_{0}^{t} \left(C \|X_s^{(n)}\|_H^{p}+F_s \cdot \|X_s^{(n)}\|_H^{p-2} +\gamma  \|X_s^{(n)}\|_H^{p-2}\|X_s^{(n)}\|_V^{\alpha}
        \right)\d s\\
        &+\frac{p}{2}\int_{0}^{t} \left(K \|X_s^{(n)}\|_H^{p} + F_s\cdot\|X_s^{(n)}\|_H^{p-2}
        \right)\d s+Y(t)+Z(t)+I(t),
    \end{align*}
    where $Y, Z, I$ are processes defined by, for $t\geq0$,
    \begin{align*}
&Y(t)=\int_{0}^{t}p\|X_{s-}^{(n)}\|_H^{p-2}\<X_{s-}^{(n)},P_n B(s,X_s^{(n)})\d
W_s^{(n)}\>_H;\\
&Z(t)=\int_{0}^{t}\int_{D^c}p\|X_{s-}^{(n)}\|^{p-2}_H\<X_{s-}^{(n)},P_nf(s,X_{s-}^{(n)},z)\>_H\tilde{N}(\d s,\d z);\\
&I(t)=\int_{0}^{t}\int_{D^c}\Big{|}\|X_{s-}^{(n)}+P_nf(s,X_{s-}^{(n)},z)\|^p_H-\|X_{s-}^{(n)}\|^p_H\\
&\hspace{1cm}-p\|X_{s-}^{(n)}\|_H^{p-2}\<X_{s-}^{(n)},P_nf(s,X_{s-}^{(n)},z)\>_H\Big{|}N(\d s,\d z).
\end{align*}
Note that $\|X_t^{(n)}\|_H\leq R$, for $t<\tau_R^{(n)}$ and $X_t^{(n)}$ takes values in $H_n$. Hence
there exists a constant $C$ such that
$$\|X_t^{(n)}\|_V\leq C R, \ \  t<\tau_R^{(n)}.$$
Thus
 by (\ref{c3}) we have
 $$\E\int_0^{t\wedge\tau_R^{(n)}}\|X_{s-}^{(n)}\|_H^{2(p-1)}\|B(s,X_{s}^{(n)})\|^2_2\ \d s<\infty$$
$$\E\int_0^{t\wedge\tau_R^{(n)}}\int_{D^c}\|X_{s-}^{(n)}\|_H^{2(p-1)}\|f(s,X_{s-}^{(n)},z)\|_H^2\ \nu(\d z)\d s<\infty.$$
Therefore, the stopped processes  $Y_{t\wedge\tau_R^{(n)}}$ and $Z_{t\wedge\tau_R^{(n)}}$ are martingales.
 Denote,  for notational simplicity, the stopped process $X_{t\wedge\tau_R^{(n)}}^{(n)}$ etc again by $X_{t }^{(n)}$ etc.
Then by  Young's inequality and martingale property  we have
\begin{align*}
      &  \E  \|X_t^{(n)}\|_H^p + \left(\frac{p\theta}{2}-\gamma p(p-2)  \right) \E \int_{0}^{t}  \|X_s^{(n)}\|_H^{p-2}\|X_s^{(n)}\|_V^{\alpha}\d s\\
    \le & \E\|x\|_H^p + C \E \int_{0}^{t}\left(\|X_s^{(n)}\|_H^{p}+ F_s^{p/2}\right)\d s  +  \E I(t),
\end{align*}
where $C$ is  some  constant.

  Define $g(t):=\|x+t h\|_H^p$,  then by applying the Taylor formula to $g$ (cf. e.g.\cite{ana+Dra}) we can get  that for some constant $C_p$ ($p\geq2$)
\begin{align}\label{eq-31}
       \left|\|x+h\|_H^{p}-\|x\|^{p}_H-p\|x\|_H^{p-2} \<x, h\>_H\right|
       &=p \left|   \int_0^1 [\|x+th\|_H^{p-2}\langle x+th,h\rangle_H-\|x\|_H^{p-2}\langle x, h\rangle_H ]\, \d t   \right|\nonumber\\
       &\leq C_p\int_0^1(\|x\|_H+\|h\|_H)^{p-2}|h|_H^2 t\,\d t\\
&\leq C_p(\|x\|^{p-2}_H\|h\|_H^2+\|h\|_H^{p}),\   x,h\in H_n.\nonumber
\end{align}
 In particular, if $p=2$, the above inequality can be replaced by the equality with $C_p=1$, i.e.
\begin{equation*}
   \left|\|x+h\|_H^{2}- \|x\|^{2}_H-2\<x,h\>_H\right|=\|h\|_H^2,\ \ \text{for all } x,h\in H_n.
\end{equation*}
Thus it follows from \eqref{eq-31} and \eqref{c4}  that
\begin{align}
    \begin{split}\label{Est-I}
 \E I(t)&\leq\E\int_{0}^{t}\int_{D^c}\Big{|}\|X_{s-}^{(n)}+P_nf(s,X_{s-}^{(n)},z)\|^p_H-\|X_{s-}^{(n)}\|^p_H\\
&\hspace{1cm}-p\|X_{s-}^{(n)}\|_H^{p-2}\<X_{s-}^{(n)},P_nf(s,X_{s-}^{(n)},z)\>_H\Big{|}N(\d s,\d z)\\
&=\E\int_{0}^{t}\int_{D^c}\Big{|}\|X_{s}^{(n)}+P_nf(s,X_{s}^{(n)},z)\|^p_H-\|X_{s}^{(n)}\|^p_H\\
&\hspace{1cm}-p\|X_{s}^{(n)}\|^{p-2}\<X_{s}^{(n)},P_nf(s,X_{s}^{(n)},z)\>_H\Big{|}\nu(\d z)\d s\\
&\leq  C \E\int_{0}^{t}  \|X_s^{(n)}\|^{p-2}_H \|f(s,X_{s}^{(n)},z)\|_H^2 \d s+   C \E\int_{0}^{t}  \|f(s,X_{s}^{(n)},z)\|_H^{p}  \d s\\
 &\leq C  \E\int_{0}^{t} \left(F_t^{p/2}+\|X_s^{(n)}\|^p_H  \right)\d s,
\end{split}
\end{align}
where  $C$ is  a generic constant.

Combining the above  estimates  we get
\begin{align*}
      &   \E  \|X_t^{(n)}\|_H^p + \left( \frac{p\theta}{2}- \gamma p(p-2)\right)  \E   \int_{0}^{t}  \|X_s^{(n)}\|_H^{p-2}\|X_s^{(n)}\|_V^{\alpha}\d s\\
\le & \E\|x\|_H^p  + C \E\int_{0}^{t} \left( \|X_s^{(n)}\|_H^p +F_s^{p/2} \right) \d s,
\end{align*}
where $C$ is some constant.

By the Gronwall
Lemma we have some constant $C>0$ such that for any $ R\ge 0$ and any $n\ge 1$
$$ \E \|X_{t\wedge\tau_R^{(n)}}^{(n)}\|_H^p +\E\int_{0}^{T\wedge\tau_R^{(n)}}  \|X_s^{(n)}\|_H^{p-2}\|X_s^{(n)}\|_V^{\alpha}\d s
\le C\left(\E\|x\|_H^p+ \E\int_{0}^{T} F_s^{p/2}\d s\right), \ t\ge 0.$$
Here the constant $C$ is independent of $n$ and the stopping times $\tau_R^{(n)}$.
Therefore, by applying the Fatou Lemma we get the desired inequality \eqref{l2'}.
\end{proof}

 If we assume the  assumptions of Theorem \ref{T1}, but with the condition \eqref{c4} replaced by a weaker assumption
 \begin{align}
& \int_{D^c} \|f(t,v,z)\|^{\beta+2}_H\nu(\d z)\leq   F_t^{(\beta+2)/2}+ C\|v\|^{\beta+2}_H+\gamma \|v\|_H^\beta\|v\|_V^\alpha, \label{c4'}
\end{align}
  we arrive at the following Lemma.
\begin{lem}\label{R1}
There exists  a constant $\gamma_0$ such that if  \eqref{c4'} is satisfied with  $\gamma<\gamma_0$,  then we have
\begin{equation}\begin{split}\label{l2}
 \sup_{n\in\mathbb{N}}\Big{(}\E\sup_{t\in[0,T]}\|X_t^{(n)}\|_H^{\beta+2} +\E\int_{0}^T \|X_t^{(n)}\|_H^{\beta}\|X_t^{(n)}\|_V^\alpha \d t\Big{)}
 \le  C\left(\E\|x \|_H^{\beta+2}+ \E\int_0^T F_t^{(\beta+2)/2}\d t\right).
\end{split}\end{equation}
\end{lem}
\begin{proof}  Let $p=\beta+2$ as before. By
  \eqref{Ito estimate 2},   $(H3)$ and \eqref{c3}, we find
\begin{align*}
      &   \sup_{s\in[0,t\wedge \tau_R^{(n)}]} \|X_s^{(n)}\|_H^p + \frac{p\theta}{2} \int_{0}^{t\wedge\tau_R^{(n)}}  \|X_s^{(n)}\|_H^{p-2}\|X_s^{(n)}\|_V^{\alpha}\d s\\
 \le & \|x\|_H^p+p(p-2)\int_{0}^{t\wedge\tau_R^{(n)}} \left(C \|X_s^{(n)}\|_H^{p}+F_s \cdot \|X_s^{(n)}\|_H^{p-2} +\gamma  \|X_s^{(n)}\|_H^{p-2}\|X_s^{(n)}\|_V^{\alpha}
        \right)\d s\\
        &+\frac{p}{2}\int_{0}^{t\wedge\tau_R^{(n)}} \left(K \|X_s^{(n)}\|_H^{p} + F_s\cdot\|X_s^{(n)}\|_H^{p-2}
        \right)\d s
     +I_1(t)+I_2(t)+I_3(t)\\
    \le & \|x\|_H^p +\gamma p(p-2) \int_{0}^{t\wedge\tau_R^{(n)}} \|X_s^{(n)}\|_H^{p-2}\|X_s^{(n)}\|_V^{\alpha}\d s\\
&+  C \int_{0}^{t\wedge\tau_R^{(n)}}\left(\|X_s^{(n)}\|_H^{p}+ F_s^{p/2}\right)\d s+I_1(t)+I_2(t)+I_3(t),
\end{align*}
where $C$ is  a generic constant, $\tau_R^{(n)}$ are the stopping times defined in (\ref{stopping time})  and $I_1, I_2, I_3$ are processes defined by, for $t\ge 0$,
\begin{align*}
        &I_1(t):=p\sup_{r\in[0,t\wedge\tau_R^{(n)}]}\left|\int_{0}^{r}\|X_s^{(n)}\|_H^{p-2}\<X_s^{(n)}, P_nB(s,X_s^{(n)})\d
W_s^{(n)}\>_H\right|;\\
&I_2(t):=p\sup_{r\in[0,t\wedge\tau_R^{(n)}]}\left|\int_{0}^{r}\int_{D^c}\|X_s^{(n)}\|^{p-2}_H\<X_{s-}^{(n)},P_nf(s,X_{s-}^{(n)},z)\>_H\tilde{N}(\d s,\d z)\right|;\\
&I_3(t):=\sup_{r\in[0,t\wedge\tau_R^{(n)}]}\Big|\int_{0}^{r}\int_{D^c}\Big{[}\|X_{s-}^{(n)}+P_nf(s,X_{s-}^{(n)},z)\|^p_H-\|X_{s-}^{(n)}\|^p_H\\
&\hspace{4cm}-p\|X_{s-}^{(n)}\|_H^{p-2}\<X_{s-}^{(n)},P_nf(s,X_{s-}^{(n)},z)\>\Big{]}N(\d z,\d s)\Big|.
\end{align*}

 On the basis of the Burkholder-Davis-Gundy inequality (cf.\cite{[Ichikawa]}), assumption \eqref{c3}, the Cauchy-Schwartz and the Young inequalities,  we have
for any $\varepsilon>0$,
\begin{align}\label{Est-I-1}
 &~~  \E I_1(t)\\  \nonumber
=&p\E\sup_{r\in[0,t\wedge\tau_R^{(n)}]}\left|\int_{0}^r \|X_s^{(n)}\|_H^{p-2} \<X_s^{(n)},P_n B(s,X_s^{(n)})\d W_s^{(n)}\>_H  \right| \\   \nonumber
   \le & 3p\E\left[\int_{0}^{t\wedge\tau_R^{(n)}}\|X_s^{(n)}\|_H^{2p-2} \|B(s,X_s^{(n)})\|_{2}^2 \d s  \right]^{1/2}\\    \nonumber
   \le & 3p\E\left[ \sup_{s\in[0,t\wedge\tau_R^{(n)}]}\|X_s^{(n)}\|_H^{p}\cdot
\Big(\int_{0}^{t\wedge\tau_R^{(n)}} \|X_s^{(n)}\|_H^{p-2}\left(F_s+C \|X_s^{(n)}\|_H^2+\gamma \|X_s^{(n)}\|_V^{\alpha}\right) \d s\Big)  \right]^{1/2}\\  \nonumber
\le & 3p\left[\varepsilon\E \sup_{s\in[0,t]}\|X_s^{(n)}\|_H^{p}\right]^{1/2}  \left[\frac{1}{\varepsilon}\E
\Big(\int_{0}^{t\wedge\tau_R^{(n)}} \|X_s^{(n)}\|_H^{p-2}\left(F_s+C \|X_s^{(n)}\|_H^2+ \gamma \|X_s^{(n)}\|_V^{\alpha}\right) \d s\Big)  \right]^{1/2} \\  \nonumber
\le &  \varepsilon \E \sup_{s\in[0,t\wedge\tau_R^{(n)}]}\|X_s^{(n)}\|_H^{p}+ C_{\varepsilon,p}\E
\Big(\int_{0}^{t\wedge\wedge\tau_R^{(n)}} \|X_s^{(n)}\|_H^{p-2}\left(F_s+ C\|X_s^{(n)}\|_H^2+ \gamma \|X_s^{(n)}\|_V^{\alpha}\right) \d s\Big)  \\  \nonumber
\le & \varepsilon  \E\sup_{s\in[0,t\wedge\tau_R^{(n)}]}\|X_s^{(n)}\|_H^{p}+\gamma C_{\varepsilon,p}\E\int_{0}^{t\wedge\tau_R^{(n)}}\|X_s^{(n)}\|_H^{p-2}\|X_s^{(n)}\|_V^{\alpha}ds\\
    &+ C_{\varepsilon,p} \E\int_{0}^{t\wedge\tau_R^{(n)}} \left(\|X_s^{(n)}\|_H^p+
F_s^{p/2} \right)\d s.  \nonumber
\end{align}
 Similarly,  using   (\ref{c3}), the Burkholder-Davis inequality  and  the Young inequality  we have
\begin{align}\label{Est-I-2}
\begin{aligned}
    &~~~~ \E I_2(t)\\
&=p\,\E\sup_{r\in[0,t\wedge\tau_R^{(n)}]}\left| \int_{0}^r\int_{D^c}\|X_s^{(n)}\|^{p-2}_H\<X_{s-}^{(n)},P_nf(s,X_{s-}^{(n)},z)\>_H\tilde{N}(\d s,\d z) \right|\\
    &\leq C\E\left[\int_{0}^{t\wedge\tau_R^{(n)}}\int_{D^c}\|X_{s}^{(n)}\|^{2p-2}_H\|P_nf(s,X_{s}^{(n)},z)\|^2_H\nu(\d z)\d s\right]^{\frac{1}{2}}\\
     &\leq C\E\left[\sup_{s\in[0,t\wedge\tau_R^{(n)}]} \|X_{s}^{(n)}\|^{p}_H\left(\int_{0}^{t\wedge\tau_R^{(n)}}\|X_{s}^{(n)}\|^{p-2}_H(F_s+C\|X_{s}^{(n)}\|^2_H+ \gamma
\|X_{s}^{(n)}\|_V^{\alpha})\d s\right)\right]^{\frac{1}{2}}\\
     &\leq \varepsilon \E\sup_{s\in[0,t\wedge\tau_R^{(n)}]} \|X_s^{(n)}\|^p_H   +C_{\varepsilon,p}\E
    \Big(\int_{0}^{t\wedge\tau_R^{(n)}} \|X_{s}^{(n)}\|^{p-2}\left(F_s+\|X_{s}^{(n))}\|_H^2+\gamma\|X_{s}^{(n))}\|_V^{\alpha}\right) \d s\Big) \\
     & \leq  \varepsilon \E\sup_{s\in[0,t\wedge\tau_R^{(n)}]}\|X_s^{(n)}\|_H^{p} + \gamma C_{\varepsilon,p}\E\int_{0}^{t\wedge\tau_R^{(n)}}\|X_{s}^{(n)}\|_H^{p-2}
\|X_{s}^{(n)}\|_V^{\alpha}ds\\
&\ + C_{\varepsilon,p} \E\int_{0}^{t\wedge\tau_R^{(n)}} \left(\|X_{s}^{(n)}\|_H^p+
    F_s^{p/2} \right)\d s,
    \end{aligned}
\end{align}
where  $C_{\varepsilon,p}$ is not necessarily the same number from line to line.

 For the term $I_3(t)$, by \eqref{c4'}, \eqref{c3} and \eqref{eq-31}, we have
\begin{align}\label{Est-I-3}
 \E I_3(t)&\leq\E\int_{0}^{t\wedge\tau_R^{(n)}}\int_{D^c}\Big{|}\|X_{s-}^{(n)}+P_nf(s,X_{s-}^{(n)},z)\|^p_H-\|X_{s-}^{(n)}\|^p_H\\  \nonumber
&\hspace{1cm}-p\|X_{s-}^{(n)}\|_H^{p-2}\<X_{s-}^{(n)},P_nf(s,X_{s-}^{(n)},z)\>_H\Big{|}N(\d s,\d z)\\  \nonumber
&=\E\int_{0}^{t\wedge\tau_R^{(n)}}\int_{D^c}\Big{|}\|X_{s}^{(n)}+P_nf(s,X_{s}^{(n)},z)\|^p_H-\|X_{s}^{(n)}\|^p_H\\  \nonumber
&\hspace{1cm}-p\|X_{s}^{(n)}\|^{p-2}\<X_{s}^{(n)},P_nf(s,X_{s}^{(n)},z)\>_H\Big{|}\nu(\d z)\d s\\  \nonumber
&\leq C_p\E\int_{0}^{t\wedge\tau_R^{(n)}}\int_{D^c}\Big{(}\|X_{s}^{(n)}\|_H^{p-2}\|f(s,X_{s}^{(n)},z)\|^2_H+\|f(s,X_{s}^{(n)},z)\|_H^p\Big{)}\nu(\d z)\d s\\  \nonumber
&\leq \gamma C_p \E\int_{0}^{t\wedge\tau_R^{(n)}}  \|X_{s}^{(n)}\|^{p-2}_H\|X_{s}^{(n)}\|_V^{\alpha}\d s+
 C_p  \E\int_{0}^{t\wedge\tau_R^{(n)}}(F_t^{p/2}+\|X_{s}^{(n)}\|^p_H )\d s.
\end{align}
Combining the  estimates \eqref{Est-I-1}-\eqref{Est-I-3} we get
\begin{align*}
    &\E (I_1(t)+I_2(t)+I_3(t))\\
    \leq & 2 \varepsilon \E\sup_{s\in[0,t\wedge\tau_R^{(n)}]}\|X_s^{(n)}\|_H^{p}+
\gamma C_{\varepsilon,p}\E\int_{0}^{t\wedge\tau_R^{(n)}}\|X_s^{(n)}\|_H^{p-2}\|X_s^{(n)}\|_V^{\alpha}ds\\
 & \ \ \ + C_{\varepsilon,p} \E\int_{0}^{t\wedge\tau_R^{(n)}} \|X_{s}^{(n)}\|_H^p\d s  +C_{\varepsilon,p}\E\int_{0}^{T}
    F_s^{p/2} \d s.
\end{align*}
Let $\varepsilon=\frac{1}{3}$,  then we have
\begin{align*}
      &  \frac{1}{3} \E \sup_{s\in[0,t\wedge\tau_R^{(n)}]} \|X_s^{(n)}\|_H^p + \left( \frac{p\theta}{2}-3\gamma C_0\right)
 \E   \int_{0}^{t\wedge\tau_R^{(n)}}  \|X_s^{(n)}\|_H^{p-2}\|X_s^{(n)}\|_V^{\alpha}\d s\\
\le & \E\|x\|_H^p  + C_{0} \E\int_{0}^{t\wedge\tau_R^{(n)}} \|X_{s}^{(n)}\|_H^p\d s  +C_{0}\E\int_{0}^{T}
    F_s^{p/2} \d s,
\end{align*}
where $C_0$ is some constant.

Observe that $\|X_{s}^{(n)}\|_H\leq R$, for $s<\tau_R^{(n)}$. Then we see that the right-hand side of the above inequality is finite.
 Therefore, if $\gamma$ is  small enough  (e.g.  $\gamma<\gamma_0:=\frac{p\theta}{6C_0}$), we may apply  the Gronwall
Lemma
 to infer that there exists $C>0$ such that for any $n\ge 1$
$$\E\sup_{t\in[0, T\wedge\tau_R^{(n)}]}\|X_t^{(n)}\|_H^p +\E\int_{0}^{T\wedge\tau_R^{(n)}}  \|X_s^{(n)}\|_H^{p-2}\|X_s^{(n)}\|_V^{\alpha}\d s
\le C\left(\E\|x\|_H^p+ \E\int_{0}^{T} F_s^{p/2}\d s\right).$$
Recall that $\tau_R^{(n)}\uparrow T$
, $\mathbb{P}$-a.s. and $\mathbb{P}\{\tau_R^{(n)}<T\}=0$ as $R\rightarrow\infty$.
           Therefore, by the  Fatou  Lemma we obtain that
\begin{align*}
             &\E\sup_{t\in[0, T]}\|X_t^{(n)}\|_H^p +\E\int_{0}^{T}  \|X_s^{(n)}\|_H^{p-2}\|X_s^{(n)}\|_V^{\alpha}\d s\\
            &\leq\liminf_{R\rightarrow\infty} \left( \E\sup_{t\in[0, T\wedge\tau_R^{(n)}]}\|X_t^{(n)}\|_H^p +\E\int_{0}^{T\wedge\tau_R^{(n)}}  \|X_s^{(n)}\|_H^{p-2}\|X_s^{(n)}\|_V^{\alpha}\d s \right) \\
    &\le C\left(\E\|X_0\|_H^p+ \E\int_{0}^{T} F_s^{p/2}\d s\right),  \ \text{for all }n\geq1.
\end{align*}
 This completes the proof of Lemma \ref{R1}.
 \end{proof}

    For the  simplicity of notations,  we introduce the following three auxiliary spaces:
\begin{equation*}
\begin{split}
 &K=L^\alpha([0,T]\times \Omega, \d t\times \P;V);\\
&J=L^2([0,T]\times \Omega, \d t\times \P; \mathcal{T}_2(U;H));\\
&\mathcal{M}=\mathcal{M}^2_T(\mathcal{P}\otimes\mathcal{Z},\d t\times\mathbb{P}\times\nu;H).
\end{split}
\end{equation*}
Note that  $K^*=L^{\frac{\alpha}{\alpha-1}}([0,T]\times\Omega,\d t\times\P;V^*)$.

\begin{lem}\label{pro1}  Under the  assumptions of Theorem \ref{T1},
there exists a
subsequence $(n_k)$ and an element $\bar{X}\in K\cap L^\infty([0,T]; L^p(\Omega; H))$ such that

(i) $X^{(n_k)}\rightarrow \bar{X}$ weakly in $K$ and weakly star in
$L^\infty([0,T]; L^p(\Omega; H))$;

(ii) $Y^{(n_k)}:=P_{n_k}A(\cdot,X^{(n_k)})\rightarrow Y$ weakly in $K^*$;

(iii) $Z^{(n_k)}:=P_{n_k} B(\cdot,X^{(n_k)})\rightarrow Z$ weakly in
$J$ and \dela{hence}
$$    \int_{0}^\cdot P_{n_k} B(s,X^{(n_k)}_s)\d W_s^{(n_k)}\rightarrow \int_{0}^\cdot Z_s \d W_s     $$
weakly in $L^\infty([0,T], \d t; L^2(\Omega,\P; H)) $;

(iv) $F^{(n_k)}:=P_{n_k}f(\cdot,X^{(n_k)}, \cdot)1_{D^c}\rightarrow F1_{D^c}$ weakly in $\mathcal{M}$.
\end{lem}

\begin{proof}
    Applying Lemma \ref{L2} with $p=2$  (instead of taking $p=\beta+2$)  we have
    \begin{equation}
    \begin{split}\label{eq-10}
      &\sup_{n}\mathbb{E}\int_{0}^T\|X_t^{(n)}\|_{V}^{\alpha}dt<\infty.
      \end{split}
    \end{equation}
    Since the space K is reflexive,   we can find a weakly convergent subsequence $\{X^{(n_k)}\}$ and
 $\bar{X}\in K$ such that $X^{(n_k)}$ converges to $\bar{X}$ weakly in  $K$.

    Similarly, since $L^\infty([0,T]; L^p(\Omega; H))=(L^1([0,T];L^{\frac{p}{p-1}}(\Omega; H)))^*$,
 by the Banach-Alaoglu Theorem, \eqref{l2'}  allows us to get another weakly$^*$  convergent subsequence (for
 simplicity we still denote it by the same notation $\{X^{(n_k)}\}$)  and  $\bar{X}\in K\cap L^p(\Omega;L^{\infty}([0,T];H))$
such that assertion $(i)$ holds.
   Meanwhile, by  $(H4)$ and \eqref{l2'} we have
    \begin{align*}
     &  \sup_n\E\int_{0}^T\|A(t,X_t^{(n)})\|_{V^*}^{\frac{\alpha}{\alpha-1}} \d t\\
 \leq& \sup_n\E\int_{0}^T(F_t+ C \|X_t^{(n)}\|_V^{\alpha})(1+\|X_t^{(n)}\|_H^{\beta})  \d t\\
 \leq&  C\sup_n\E\int_{0}^T\left(F_t+\|X_t^{(n)}\|_V^{\alpha}+F_t^{\frac{\beta+2}{2}}+\|X_t^{(n)}\|^{\beta+2}_H
+\|X_t^{(n)}\|_H^{\beta}\|X_t^{(n)}\|_{V}^{\alpha}\right)  \d t < \infty.
    \end{align*}
Therefore,  claim $(ii)$ also holds.

 Also, note  that by \eqref{c3}
\begin{align*}
    & \sup_n\E\int_{0}^T\|P_{n}B(t,X_t^{(n)})\|_{2}^2 \d t\\
\leq& \sup_n\E \int_{0}^T\left(F_t+ C \|X_t^{(n)}\|_H^2+ \gamma \|X_t^{(n)}\|^{\alpha}_V\right) \d t <\infty.
\end{align*}
Hence by taking a subsequence we have that $P_{n_k}B(t,X_t^{(n_k)})$ converges to $Z$ weakly in $J$.

Recall that $\tilde{P}_n$ is the orthogonal projection in $U$ onto $span\{g_1,\cdots,g_n\}$. Hence, without loss of generality  we can assume that
$P_{n_k}B(t,X_t^{(n_k)})\tilde{P}_n$ also converges to $Z$ weakly in $J$ .
 Since
\begin{align*}
           \int_{0}^{\cdot}P_{n}B(s,X_t^{(n_k)} ) \d W^{{n_k}}_s=\int_{0}^{\cdot}P_{n_k}B(s,X_s^{(n_k)})\tilde{P}_{n_k}\d W_s,
\end{align*}
  weak convergence is preserved under the linear  continuous maps and the map
$$I:\phi\in J\mapsto I(\phi):=\int\phi\ \d W\in L^{2}([0,T]\times\Omega; H)$$
is continuous,
 we infer that
\begin{align*}
        \int_{0}^{\cdot}P_{n}B(s,X_s^{(n_k)})\tilde{P}_n\d W_s \rightarrow  \int_{0}^{\cdot} Z_s\d W_s \  \text{weakly}.
\end{align*}
Hence  $(iii)$ holds.

 Similarly,  by  \eqref{c3} we have
    \begin{align*}
             &\sup_n\E\int_{0}^{T}\int_{D^c} \|P_nf(s,X_{s-}^{(n)},z)\|_H^2\nu(\d z) \d s\\
             \leq &\sup_n\int_{0}^T\left(F_t+ C\|X_s^{(n)}\|_H^2+ \gamma \|X_s^{(n)}\|_V^{\alpha}\right)  \d s
             < \infty,
    \end{align*}
    which yields claim $(iv)$.
\end{proof}

\begin{proof}[\textbf{Proof of Theorem \ref{T2}}]
 \textit{Existence of solutions:}  Let us define a $V^*$-valued process $X$ by \begin{equation}
 X_t:=X_0+\int_0^t Y_s \d
s+\int_0^t Z_s \d W_s+\int_0^t\int_{D^c} F(s,z)\tilde{N}(\d s,\d z)
, \ t\in [0,T].
\end{equation}
By Lemma \ref{pro1}, it is easy to see that $X$ is a $V^*$-valued modification of the $V$-valued process $\bar{X}$,
i.e. $X=\bar{X} \  \d t\times\P$-a.e.. Moreover, we have
$$\sup_{t\in[0,T]}\E\|X_t\|_H^p+\E \int_0^T \|X_t\|_V^\alpha \d t   < \infty.$$
By \cite{[Gyongy-Krylov]}, we infer that $X$ is an $H$-valued c\`{a}dl\`{a}g $\mathbb{F}$-adapted
process  satisfying
\begin{align}\label{Ito}
  \|X_t\|_H^2=&\|X_0\|_H^2+\int_0^t\Big(2 { }_{V^*}\langle  Y_s, \bar{X}_s\rangle_V+\|Z_s\|^2_2\Big)\d s+2\int_0^t\langle \bar{X}_s,Z_s\d W_s\rangle_H\\
  &+2\int_0^t\int_{D^c}\langle \bar{X}_s,F(s,z)\rangle_H\tilde{N}(\d s,\d z)
  +\int_0^t\int_{D^c}\|F(s,z)\|_H^2 N(\d s,\d z).\nonumber
\end{align}

Therefore, in order to prove that $X$ is a solution of \ref{SEE-1} it remains to verify that
\begin{align*}
 & A(\cdot,\bar{X})=Y, \   B(\cdot,\bar{X})=Z, \ \     \d t\times\P -a.e.; \\
\text{and } & f(s,\bar{X}_{s-},z)=F(s, z), \  \ \d t\times\P\times\nu-a.e. .
 \end{align*}
Define
$$ \mathcal{N}=\bigg\{\phi: \phi\ \text{is a}\ V\text{-valued}\ \mathbb{F}\text{-adapted process such that}\   \E\int_0^T\rho(\phi_s) ds<\infty  \bigg\}.   $$
For $  \phi\in K\cap\mathcal{N} \cap L^\infty([0,T]; L^p(\Omega; H))$, by applying the It\^{o}  formula to the process $X^{(n_k)}$, see Schmalfuss \cite[proof of Theorem 4.1]{Schmalfuss_1997} and
 Temam \cite{Temam_2001} (for the deterministic case), we have
\begin{align*}
 &\hspace{-1truecm}\lefteqn{e^{-\int_0^t(K+\rho(\phi_s) )\d
s}\|X_t^{(n_k)}\|_H^2}\\
=&\|X_0^{(n_k)}\|_H^2+\int_0^t
e^{-\int_0^s(K+\rho(\phi_r) )\d r} \bigg(
 2{ }_{V^*}\<A(s,X^{(n_k)}_s), X_{s-}^{(n_k)}  \>_V
 \\
 & +\|P_{n_k}B(s,X_s^{(n_k)})\tilde{P}_{n_k}\|_2^2
 -(K+\rho(\phi_s) )\|X_s^{(n_k)}\|_H^2
 \bigg)\d s \bigg]\\
 &+2\int_0^te^{-\int_0^s(K+\rho(\phi_r) )\d r}\langle X_{s-}^{(n_k)},P_{n_k}B(s, X_s^{(n_k)})\d W_s^{n_k}\rangle_H\\
 &+2 \int_0^t\int_{D^c}e^{-\int_0^s(K+\rho(\phi_r)  )\d r}\langle X_{s-}^{(n_k)},P_{n_k}f(s,X_{s-}^{(n_k)},z)\rangle_H\tilde{N}(\d s, \d z)\\
 &+\int_0^t\int_{D^c}e^{-\int_0^s(K+\rho(\phi_r)  )\d r}\| P_{n_k}f(s,X_{s-}^{(n_k)},z)\|_H^2N(\d s,\d z) .
\end{align*}
 Thus, by taking the expectation of both sides of the above equality and $(H2)$  we get
\begin{align*}
& \E\left( e^{-\int_0^t(K+\rho(\phi_s) )\d
s}\|X_t^{(n_k)}\|_H^2 \right)- \E\left(\|X_0^{(n_k)}\|_H^2\right)\\
\nonumber
 =& \E \bigg[ \int_0^t
e^{-\int_0^s(K+\rho(\phi_r) )\d r} \bigg(
 2{ }_{V^*}\<A(s,X^{(n_k)}_s), X_{s-}^{(n_k)}  \>_V \\
& +\|P_{n_k}B(s,X_s^{(n_k)})\tilde{P}_{n_k}\|_2^2
 -(K+\rho(\phi_s) )\|X_s^{(n_k)}\|_H^2
 \bigg)\d s \bigg]\\ \nonumber
 &+\E\Big[\int_0^t\int_{D^c}e^{-\int_0^s(K+\rho(\phi_r) )\d r}\| P_{n_k}f(s,X_{s-}^{(n_k)},z)\|_H^2\nu(\d z)\d s\Big]\\\nonumber
\leq& \E \bigg[ \int_0^t e^{-\int_0^s(K+\rho(\phi_r)  )\d r} \bigg(
 2{ }_{V^*}\<A(s,X^{(n_k)}_s)-A(s,\phi_s), X_s^{(n_k)}-\phi_s  \>_V \\
 \nonumber
 &+\|B(s,X_s^{(n_k)})-B(s,\phi_s)\|_2^2-(K+\rho(\phi_s))\|X_s^{(n_k)}
 -\phi_s\|_H^2\\
&+\int_{D^c}\|f(s,X_{s}^{(n_k)},z)-f(s,\phi_s,z)\|_H^2\nu(\d z) \bigg) \d s
\bigg]\\ \nonumber
&+\E \bigg[ \int_0^t e^{-\int_0^s(K+\rho(\phi_r)  )\d r} \bigg(
 2{ }_{V^*}\<A(s,X^{(n_k)}_s)-A(s,\phi_s), \phi_s  \>_V + 2{ }_{V^*}\<A(s,\phi_s), X_s^{(n_k)}
 \>_V\\ \nonumber
& -\|B(s,\phi_s)\|_2^2+2\<B(s,X_s^{(n_k)}), B(s,\phi_s)
\>_{\mathcal{T}_2(U,H)}
 -2(K+\rho(\phi_s))\<X_s^{(n_k)}, \phi_s\>_H\\\nonumber
&+(K+\rho(\phi_s) )\|\phi_s\|_H^2
+\int_{D^c} \left( 2\langle f(s,X_{s}^{(n_k)},z),f(s,\phi_s,z)\rangle_H-\|f(s,\phi_s,z)\|_H^2 \right) \nu(\d z)\bigg)\d s
\bigg]\\
\leq&\E \bigg[ \int_0^t e^{-\int_0^s(K+\rho(\phi_r)  )\d r} \bigg(
 2{ }_{V^*}\<A(s,X^{(n_k)}_s)-A(s,\phi_s), \phi_s  \>_V + 2{ }_{V^*}\<A(s,\phi_s), X_s^{(n_k)}
 \>_V\\ \nonumber
& -\|B(s,\phi_s)\|_2^2+2\<B(s,X_s^{(n_k)}), B(s,\phi_s)
\>_{\mathcal{T}_2(U,H)}
 -2(K+\rho(\phi_s))\<X_s^{(n_k)}, \phi_s\>_H\\\nonumber
&+(K+\rho(\phi_s) )\|\phi_s\|_H^2
+\int_{D^c} \left( 2\langle f(s,X_{s}^{(n_k)},z),f(s,\phi_s,z)\rangle_H-\|f(s,\phi_s,z)\|_H^2 \right) \nu(\d z)\bigg)\d s
\bigg].
\end{align*}
Hence  for any nonnegative function $\psi\in L^\infty ([0,T]; \d t)$ we have
 \begin{align}\label{e9}
    &  \E \left[ \int_0^T \psi_t  \left( e^{-\int_0^t(K+\rho(\phi_s) )\d s}\|X_t\|_H^2
- \|X_0\|_H^2\right)  \d t\right] \nonumber\\
     \leq& \liminf_{k\rightarrow\infty}\E  \left[ \int_0^T \psi_t    \left( e^{-\int_0^t(K+\rho(\phi_s) )\d
    s}\|X_t^{(n_k)}\|_H^2 - \|X_0^{(n_k)}\|_H^2\right) \d t \right] \nonumber\\
    \leq& \liminf_{k\rightarrow\infty}\E   \bigg[ \int_0^T \psi_t \bigg( \int_0^t e^{-\int_0^s(K+\rho(\phi_r)  )\d r} \bigg(
     2{ }_{V^*}\<A(s,X^{(n_k)}_s)-A(s,\phi_s), \phi_s  \>_V\nonumber\\
       & + 2{ }_{V^*}\<A(s,\phi_s), X_s^{(n_k)}
     \>_V  -\|B(s,\phi_s)\|_2^2+2\<B(s,X_s^{(n_k)}), B(s,\phi_s)
    \>_{\mathcal{T}_2
(U,H)} \\ \nonumber
    & -2(K+\rho(\phi_s))\<X_s^{(n_k)}, \phi_s\>_H +(K+\rho(\phi_s) )\|\phi_s\|_H^2\\ \nonumber
    & +\int_{D^c} \left( 2\langle f(s,X_{s}^{(n_k)},z),f(s,\phi_s,z)\rangle_H-\|f(s,\phi_s,z)\|_H^2 \right) \nu(\d z)\bigg)\d s
     \bigg)  \d t \bigg]\nonumber\\
    =&\E \bigg[ \int_0^T \psi_t \bigg(  \int_0^t e^{-\int_0^s(K+\rho(\phi_r)  )\d r} \bigg(
     2{ }_{V^*}\<Y_s-A(s,\phi_s), \phi_s  \>_V \nonumber\\ \nonumber
    & + 2{ }_{V^*}\<A(s,\phi_s),\bar{X}_s
     \>_V -\|B(s,\phi_s)\|_2^2+2\<Z_s, B(s,\phi_s)
    \>_{\mathcal{T}_2(U,H)}
     \\\nonumber
    &-2(K+\rho(\phi_s))\<\bar{X}_s, \phi_s\>_H +(K+\rho(\phi_s) )\|\phi_s\|_H^2 \\ \nonumber
    & +\int_{D^c} \left( 2\langle F(s,z),f(s,\phi_s,z)\rangle_H-\|f(s,\phi_s,z)\|_H^2 \right) \nu(\d z)\bigg)\d s \bigg) \d t
    \bigg].
 \end{align}
On the other hand,  by equality \eqref{Ito} we have for  $  \phi\in K\cap\mathcal{M} \cap L^\infty([0,T]; L^p(\Omega; H))$,
\begin{align}\label{e3}
& \E\left( e^{-\int_0^t(K+\rho(\phi_s) )\d s}\|X_t\|_H^2
\right)- \E\left(\|X_0\|_H^2\right)\\ \nonumber
 =& \E \bigg[ \int_0^t
e^{-\int_0^s(K+\rho(\phi_r) )\d r} \bigg(
 2{ }_{V^*}\<Y_s, \bar{X}_s  \>_V  +\|Z_s\|_2^2 \\ \nonumber
& -(K+ \rho(\phi_s) )\|X_s\|_H^2
    +\int_{D^c} \|F(s,z)\|_H^2 \nu(\d z)\bigg)\d s
\bigg].
\end{align}
Combining  (\ref{e3}) with  (\ref{e9})  we have
\begin{align} \label{e30}
& \E\bigg[ \int_0^T \psi_t \bigg( \int_0^t e^{-\int_0^s(K+\rho(\phi_r)  )\d r}\bigg(
 2{ }_{V^*}\<Y_s-A(s,\phi_s), \bar{X}_s- \phi_s  \>_V  \\
&  ~~~~ -(K+\rho(\phi_s) )\|\bar{X}_s- \phi_s\|_H^2  +\|B(s,\phi_s)-Z_s\|_2^2 \nonumber \\
& ~~~~  +\int_{D^c} \|f(s,\phi_s,z)-F(s,z)\|^2_H\nu(\d z)\bigg)\d s \bigg) \d t
 \bigg]\leq 0.  \nonumber
\end{align}
Therefore, if we put $\phi=\bar{X}$ in \eqref{e30},
 we can obtain that $Z=B(\cdot,\bar{X})$ in $J$ and $F(\cdot,\cdot)=f(\cdot,\bar{X}_\cdot,\cdot)$ in $\mathcal{M}$.

Note that (\ref{e30}) also implies that
\begin{align} \begin{split}\label{e31}
 \E\bigg[ \int_0^T \psi_t \bigg( \int_0^t e^{-\int_0^s(K+\rho(\phi_r)  )\d r} &\bigg(
 2{ }_{V^*}\<Y_s-A(s,\phi_s), \bar{X}_s- \phi_s  \>_V \\
&  -(K+\rho(\phi_s) )\|\bar{X}_s- \phi_s\|_H^2  \bigg)\d s\bigg) \d t
 \bigg]\leq 0.
\end{split}
\end{align}

Put $\phi=\bar{X}-\varepsilon\tilde{\phi} v$ in \eqref{e31}  for
$\tilde{\phi}\in L^\infty([0,T]\times\Omega; \d t\times\P;
\mathbb{R})$ and $v\in V$,  divide both sides by $\varepsilon$ and let $\varepsilon\rightarrow 0$. Then we have
 \begin{align*}
       \E\bigg[  \int_0^T \psi_t \bigg( \int_0^t e^{-\int_0^s(K+\rho(\bar{X}_r)  )\d r}\bigg(
     2\tilde{\phi}_s { }_{V^*}\<Y_s-A(s,\bar{X}_s),  v \>_V
 \bigg)\d s \bigg) \d t \bigg]\leq 0.
 \end{align*}
Hence, we infer  $ Y=A(\cdot, \bar{X}) $.

Therefore, we conclude that the process $X=\{X_t\}_{t\geq 0}$ is a solution to
\eqref{SEE-1}. Furthermore, the estimates \eqref{e0} and \eqref{e1}
can be proved for $\{X_t\}$ by the same arguments in Lemmas \ref{L2}
and \ref{R1}.

 \textit{Uniqueness of solutions:}  we finally proceed to show the uniqueness of  solutions to problem (\ref{SEE-1}).

 Suppose that $X=(X_t)$ and $Y=(Y_t)$ are the solutions of (\ref{SEE-1}) with initial conditions
$X_0,Y_0$ respectively, i.e.
 \begin{equation}\begin{split}
 X_t&=X_0+\int_0^t A(s,X_s) \d s+\int_0^t B(s,X_s) \d W_s+\int_0^t\int_{D^c} f(s,X_{s-},z)\tilde{N}(\d s,\d z),
 \ t\in[0,T];\\
  Y_t&=Y_0+\int_0^t A(s,Y_s) \d s+\int_0^t B(s,Y_s) \d W_s+\int_0^t\int_{D^c} f(s,Y_{s-},z)\tilde{N}(\d s,\d z), \ t\in[0,T].
\end{split}
\end{equation}
We define  the following stopping times:
$$\sigma_N:=\inf\{t\in[0,T]: \|X_t\|_H\ge N\}\wedge \inf\{t\in[0,T]: \|Y_t\|_H\ge N\}\wedge T.
$$
Applying again the Schmalfuss \cite{Schmalfuss_1997} trick, by means of the  It\^{o} formula \eqref{Ito}  we have
\begin{align*}
 &e^{-\int_0^{t\wedge\sigma_N}(K+\rho(Y_s) )\d
s}\|X_{t\wedge\sigma_N}-Y_{t\wedge\sigma_N}\|_H^2 - \|X_0-Y_0\|_H^2  \\
=& \int_0^{t\wedge\sigma_N}
e^{-\int_0^s(K+\rho(Y_r) )\d r} \bigg(
 2{ }_{V^*}\<A(s,X_s)-A(s,Y_s), X_s-Y_s  \>_V
 \\
 & +\|B(s,X_s)-B(s,Y_s)\|_2^2
 -(K+\rho(Y_s) )\|X_s-Y_s\|_H^2
 \bigg)\d s \\
 &+2\int_0^{t\wedge\sigma_N}e^{-\int_0^s(K+\rho(Y_r)  )\d r}\langle X_s-Y_s,B(s, X_s)\d W_s-B(s,Y_s)\d W_s\rangle_H\\
 &+2 \int_0^{t\wedge\sigma_N}\int_{D^c}e^{-\int_0^s(K+\rho(Y_r) )\d r}\langle X_s-Y_s,f(s,X_{s-},z)-f(s,Y_{s-},z)\rangle_H\tilde{N}(\d s, \d z)\\
 &+\int_0^{t\wedge\sigma_N}\int_{D^c}e^{-\int_0^s(K+\rho(Y_r) )\d r}\| f(s,X_{s-}, z)-f(s,Y_{s-},z)\|_H^2N(\d s,\d z) .
\end{align*}
It then follows from assumption $(H2)$ that
\begin{align*}
 &\E \left[ e^{-\int_0^{t\wedge \sigma_N}(K+\rho(Y_s)  )\d
s}\|X_t-Y_t\|_H^2 \right] - \E\|X_0-Y_0\|_H^2\\
=& \E\bigg[\int_0^{t\wedge\sigma_N}
e^{-\int_0^s(K+\rho(Y_r) )\d r} \bigg(
 2{ }_{V^*}\<A(s,X_s)-A(s,Y_s), X_s-Y_s  \>_V
 \\
 &+\|B(s,X_s)-B(s,Y_s)\|_2^2
 -(K+\rho(Y_s) )\|X_s-Y_s\|_H^2
 \\
 &+\int_{D^c}\| f(s,X_{s-}, z)-f(s,Y_{s-},z)\|_H^2\nu(\d z)\bigg)\d s \bigg]
 \leq 0 .
\end{align*}
Hence if $X_0=Y_0\  \P \text{-a.s.}$, then
$$ \E\left[ e^{-\int_0^{t\wedge \sigma_N}(K+\rho(Y_s))\d s}\|X_t-Y_t\|_H^2 \right] =0, \  t\in[0,T].   $$
Note that by   (\ref{c5}) and (\ref{e0}) (see Lemma  \ref{L2}) we have
$$ \int_0^T (K+\rho(Y_s))\d s< \infty, \ \P\text{-a.s.}.  $$
Therefore, by letting $N\rightarrow\infty$ (hence $\sigma_N\uparrow T$) we have that $  X_t=Y_t, \ \P\text{-a.s.}, t\in [0,T]. $
Then  the pathwise uniqueness follows from the path c\`{a}dl\`{a}g property of $X,Y$ in $H$.

This completes the proof of Theorem \ref{T2}.
        \end{proof}

  \subsection{With large jumps}

Let $\tau$ be a stopping time such that $\tau<\infty$ a.s.. We  define
\begin{align}
    \begin{split}\label{BM}
    W^{\tau}(t)&=W(t+\tau)-W(\tau); \\
     p^\tau(t)&=p(t+\tau), t\in\mathcal{D}(p^{\tau}),
    \end{split}
    \end{align}
    where $\mathcal{D}(p^{\tau})=\{t\in(0,\infty):  t+\tau\in\mathcal{D}(p)\}$.
 Let $\mathcal{F}^{\tau}_t=\mathcal{F}_{t+\tau}$, $t\in[0, T-\tau]$. The following result is a direct
 extension of Theorem II6.4 and II6.5 in \cite{[Ikeda]}.
\begin{prp}
    The process $W^{\tau}$ defined by \eqref{BM} is a cylindrical $\mathcal{F}_t^{\tau}$-Wiener process and
$p^{\tau}$ is a stationary $\mathcal{F}^{\tau}_t$-Poisson point process with the intensity measure $\nu$.
\end{prp}

Clearly $W^{\tau}$ is independent of $\mathcal{F}_{\tau}$ and $W^{\tau}, p^{\tau}$ enjoy the same properties as $W, p$.

\begin{cor}\label{c6} Let $\tau$ be  an $[0,T]$-valued stopping time on  and $X_{\tau}$ be an $\mathcal{F}_{\tau}$-measurable random variable.
Under the assumptions of Theorem \ref{T1}, there exists a
unique c\`{a}dl\`{a}g $H$-valued $\mathbb{F}$-adapted process
$(X_t)$ and a process $\bar{X}\in L^\alpha([\tau,T]; V)\cap
L^2([\tau,T]; H)$, $\mathbb{P}$-a.s.
 which is $\d t\times\mathbb{P}$-equivalent to $X$ such that the equality holds $\mathbb{P}$-a.s.:
          \begin{align}\label{SE1}
                      X_t=X_{\tau}+\int_{\tau}^tA(s,\bar{X}_s)\d s+\int_{\tau}^t B(s,\bar{X}_s)\d W_s+\int_{\tau}^t\int_{D^c} f(s,\bar{X}_{s-},z)\tilde{N}(\d s,\d z),\ t\in[\tau,T].
          \end{align}
    Moreover, if $X_\tau\in L^{\beta+2}(\Omega,\mathcal{F}_\tau, \P; H)$, then we have
$$\bar{X}\in L^\alpha([\tau,T]\times \Omega, \d t \times\P; V)\cap L^{\beta+2}([\tau,T]\times \Omega, \d t \times\P; H).$$
\end{cor}
\begin{proof} We first assume $X_\tau=h\in H$, then it is obvious that $X_\tau\in L^{\beta+2}(\Omega,\mathcal{F}_\tau, \P; H)$.

Let $N^{\tau}$ be the compensated Poisson random measure associated to the Poisson point process $p^{\tau}$.
 As an immediate consequence of Theorem \ref{T2}, there exists a unique  $(\mathcal{F}^{\tau}_t)$-adapted
 $H$-valued c\`{a}dl\`{a}g  process $X^{\tau,h}$ such that
    \begin{align*}
        X^{\tau,h}_t=& h+ \int_{0}^tA(s+\tau,\bar{X}_s^{\tau,h})ds+\int_{0}^tB(s+\tau,\bar{X}_s^{\tau,h})\d W^{\tau}_s\\
&+\int_{0}^t\int_{D^c}f(s+\tau,\bar{X}_{s-}^{\tau,h},z)\tilde{N}^{\tau}(\d
s,\d z), \   t\in[0, T-\tau],
    \end{align*}
where as before $\bar{X}^{\tau,h}$ is the $\d t\times\mathbb{P}$-equivalent class of $X^{\tau,h}$.
 Indeed, this follows along the same lines of the proof of Theorem \ref{T2} in such a way that all computations
 involving the expectations are replaced by conditional expectations with respect to $\mathcal{F}_{\tau}$.

 Since for any $h\in H$,  the solution $X^{\tau,h}_t$ is a measurable function of $h$, by replacing $h$ with the $\mathcal{F}_\tau$-measurable
random variable $X_\tau$, where $X_{\tau}$, $W^{\tau}$ and $p^{\tau}$ are mutually independent, we obtain an unique solution $X^{\tau}$ satisfying
       \begin{align*}
        X^{\tau}_t=& X_{\tau}+ \int_{0}^tA(s+\tau,\bar{X}_s^{\tau})ds+\int_{0}^tB(s+\tau,\bar{X}_s^{\tau})\d W^{\tau}_s\\
&+\int_{0}^t\int_{D^c}f(s+\tau,\bar{X}_{s-}^{\tau},z)\tilde{N}^{\tau}(\d
s,\d z), \   t\in[0, T-\tau].
    \end{align*}
 Set $X_t:=X_{t-\tau}^{\tau}$ and $\bar{X}_t:=\bar{X}^{\tau}_{t-\tau}$, for $t\in[\tau,T]$,  then
it is straightforward to see that $X$ satisfies Equation \eqref{SE1} with the initial condition $X_\tau$.
\end{proof}

For convenience, we use   $X'_{\tau,t}(\xi)$, $t\in[\tau, T]$ to
denote the solution to Equation \eqref{SE1} on $[\tau,T]$ with
initial condition $\xi$ at time $\tau$ and $X_{0,t}(x)$, $t\in[0,T]$
to denote the solution to equation \eqref{SEE} on $[0,T]$ with
initial condition $x$ at time $0$.

 Theorem \ref{T2} tells us that equation (\ref{SEE-1})  with initial condition $x$ at time  $0$ has a unique $H$-valued
c\`{a}dl\`{a}g solution $X':=(X'_{0,t}(x))_{t\in[0,T]}$ on $[0,T]$, that is
\begin{align*}
    X'_{0,t}(x)=&x+\int_0^tA(s,\bar{X}'_{0,s}(x))\d s+\int_0^tB(s,\bar{X}'_{0,s}(x)) \d W_s\\
& +\int_0^t\int_{D^c}f(s,\bar{X}'_{0,s-}(x),z)\tilde{N}(\d s, \d z), \  t\in[0,T].
\end{align*}
 Here $\bar{X}'_{0,\cdot}(x)\in L^{\alpha}([0,T]\times\Omega,\d t\times\mathbb{P};V)\cap L^2([0,T]\times\Omega, \d t\times\mathbb{P};H)$
 and it is $\d t\times\mathbb{P}$-equivalent to $X'_{0,\cdot}(x)$.
Recall that $\{\tau_n\}$ are the arrival times for the jumps of the Poisson process $N(t,D)$, $t\in[0,T]$.
Now we may construct a solution to \eqref{SEE} on $[0,\tau_1]$ as follows:
\begin{align*}
     X_{0,t}(x)=\left\{\begin{array}{cc}
                                    X'_{0,t}(x), & \text{for } 0\leq t<\tau_1;\\
                                    X'_{0,\tau_1-}(x)+g(\tau_1,\bar{X}'_{0,\tau_1-}(x),p(\tau_1)), & \text{for } t=\tau_1.
     \end{array}
    \right.
\end{align*}
We note that since the process $X_{0,t}'(x), t\in[0,T]$ has no jumps
occurring at time $\tau_1$, we infer
$X_{0,\tau_1-}(x)=X'_{0,\tau_1-}(x)=X'_{0,\tau_1}(x)$. Set
$\bar{X}_{0,t}(x)=\bar{X}'_{0,t}(x)$ on $[0,\tau_1)$ and
$\bar{X}_{0,\tau_1}(x)=\bar{X}'_{0,\tau_1-}(x)+g(\tau_1,\bar{X}'_{0,\tau_1-}(x),p(\tau_1))$.
It clearly forces that $\bar{X}_{0,t}(x)$ is $\d
t\times\mathbb{P}$-equivalent to $X_{0,t}(x)$ on $[0,\tau_1]$. Hence
we have
 \begin{align*}
    X_{0,\tau_1}(x)&= X'_{0,\tau_1-}(x)+g(\tau_1,\bar{X}'_{0,\tau_1-}(x),p(\tau_1))\\
&= x+\int_0^{\tau_1}A(s,\bar{X}_{0,s}(x))\d s+\int_0^{\tau_1}B(s,\bar{X}_{0,s}(x))\d W_s\\
&\hspace{1cm}+\int_0^{\tau_1}\int_{D^c}f(s,\bar{X}_{0,s-}(x),z)\tilde{N}(\d s,\d z) +g(\tau_1,\bar{X}_{0,\tau_1-}(x),p(\tau_1)).
 \end{align*}
Also, since $\tau_1$ is the first jump time of the process $N(t,D)$, $t\in[0,T]$, we infer that
\begin{align*}
          \int_{0}^t\int_{D}g(s,\bar{X}'_{0,s-}(x),z)N(\d s,\d z)=\left\{
    \begin{array}{cc}
                    0, & \ t\in[0,\tau_1),   \\
                     g(\tau_1,\bar{X}'_{0,\tau_1-}(x),p(\tau_1)), &\ t\in[\tau_1,\tau_2).
    \end{array}
    \right.
\end{align*}
It follows that for $t\in[0,\tau_1]$ we have
\begin{align*}
     X_{0,t}(x)=&  x+\int_0^{t}A(s,\bar{X}_{0,s}(x))\d s+\int_0^{t}B(s,\bar{X}_{0,s}(x))\d W_s\\
    &+\int_0^{t}\int_{D^c}f(s,\bar{X}_{0,s-}(x),z)\tilde{N}(\d s, \d z)+ \int_{0}^{t}\int_{D}g(s,\bar{X}_{0,s-}(x),z)N(\d s,\d z),
\end{align*}
which shows that
the process $X_{0,t}(x)$ is an $H$-valued solution to the equation \eqref{SEE} on $[0,\tau_1]$.

Since the valued of $g(\cdot,X,\cdot)$ at time $\tau_1$ depends only on the valued of $X_{\tau_1-}$ strictly prior to the time $\tau_1$,
the uniqueness of the solution $X'_{0,t}(x)$ on $[0,\tau_1)$ implies the uniqueness of the solution $X_{0,t}(x)$ on $[0,\tau_1]$.

 According to Corollary \ref{c6}, let us denote $X'_{\tau_1,t}(X_{0,\tau_1}(x))$  the unique solution to the equation \eqref{SEE-1} with initial condition
 $X_{0,\tau_1}(x)$ at time $\tau_1$,  then  there exists a $\d t\times\mathbb{P}$-equivalent class $\bar{X}'_{\tau_1,t}(X_{0,\tau_1}(x))$, $t\in[\tau_1,T]$ satisfying
    \begin{align*}
        X'_{\tau_1,t}(X_{0,\tau_1}(x))=    X_{0,\tau_1}(x)&+\int_{\tau_1}^tA(s,\bar{X}'_{\tau_1,s}(X_{0,\tau_1}(x)))\d s+\int_{\tau_1}^tB(s,\bar{X}'_{\tau_1,s}(X_{0,\tau_1}(x)))\d W_s\\
        &+\int_{\tau_1}^t\int_{D^c}f(s,\bar{X}'_{\tau_1,s-}(X_{0,\tau_1}(x)),z)\tilde{N}(\d s,\d z),\ \ t\in[\tau_1,T].
    \end{align*}
 We define
\begin{align*}
       X_{0,t}(x) =\left\{\begin{array}{cc}
                                    X_{0,t}(x), & \text{for } 0\leq t\leq \tau_1;\\
                                    X'_{\tau_1,t}(X_{0,\tau_1}(x)), & \text{for } \tau_1<t<\tau_2;\\
                                    X'_{\tau_1,\tau_2-}(X_{0,\tau_1}(x))+g(\tau_2,\bar{X}'_{\tau_1,\tau_2-}(X_{0,\tau_1}(x)),p(\tau_2)),  & \text{for } t=\tau_2,
     \end{array}
     \right.
\end{align*}
and
                   \begin{align*}
    \bar{X}_{0,t}(x) =\left\{
    \begin{array}{cc}
                \bar{X}_{0,t}(x), & \text{ for } 0\leq t\leq \tau_1;; \\
        \bar{X}'_{\tau_1,t}(X_{0,\tau_1}(x)), & \text{ for } \tau_1<t<\tau_2\\
        \bar{X}'_{\tau_1,\tau_2-}(X_{0,\tau_1}(x))+g(\tau_2,\bar{X}'_{\tau_1,\tau_2-}(X_{0,\tau_1}(x)),p(\tau_2)),  & \text{for } t=\tau_2,
    \end{array}
    \right.
\end{align*}
Clearly, $\bar{X}_{0,s}(x)=X_{0,s}(x), \d t
\times\mathbb{P}$ on $[0,\tau_2]$.
Then we have for $t\in(\tau_1,\tau_2)$,
\begin{align*}
      X_{0,t}(x)=& X'_{\tau_1,t}(X_{0,\tau_1}(x))\\
    =&x+\int_0^{t}A(s,\bar{X}_{0,s}(x))\d s+\int_0^{t}B(s,\bar{X}'_{0,s}(x))\d W_s\\
        &+\int_0^{t}\int_{D^c}f(s,\bar{X}_{0,s-}(x),z)\tilde{N}(\d s,\d z)+\int_{0}^t\int_{D}g(s,\bar{X}_{0,s-}(x),z)N(\d s,\d z).
\end{align*}

 As we known that $X_{\tau_1,\tau_2-}'(X_{0,\tau_1}(x))=X'_{\tau_1,\tau_2}(X_{0,\tau_1}(x))$, a similar argument as above gives
\begin{align*}
     X_{0,\tau_2}(x) =&x+\int_0^{\tau_2}A(s,\bar{X}_{0,s}(x))\d s+\int_0^{\tau_2}B(s,\bar{X}_{0,s}(x))\d W_s\\
&+\int_0^{\tau_2}\int_{D^c}f(s,\bar{X}_{0,s-}(x),z)\tilde{N}(\d s,\d z)
        +\int_{0}^{\tau_2}\int_{D}g(s,\bar{X}_{0,s-}(x),z)N(\d s,\d z).
\end{align*}
In particular,
\begin{align*}
       \int_{0}^{\tau_2}\int_{D}g(s,\bar{X}_{0,s-}(x),z)N(\d s,\d z)&=g(\tau_1,\bar{X}_{0,\tau_1-},p(\tau_1))+ g(\tau_2,\bar{X}_{0,\tau_2-},p(\tau_2))\\
    &=g(\tau_1,\bar{X}'_{0,\tau_1-}(x),p(\tau_1))+ g(\tau_2,\bar{X}'_{\tau_1,\tau_2-}(X_{0,\tau_1}(x)),p(\tau_2)) .
\end{align*}
Therefore, $X_{0,t}(x)$ is a solution of  \eqref{SEE} on $[0,\tau_2]$ and the uniqueness of the solution  on $[0,\tau_2]$
follows from the uniqueness of the solutions $X'_{0,t}(x)$ and $X'_{\tau_1,t}(X_{0,\tau_1})(x)$.

By using this type of interlacing structure, one can  construct a unique solution recursively to the equation \eqref{SEE}
in the time interval $[0,\tau_n]$ for every $n\in\mathbb{N}$.

Now the proof of Corollary \ref{c6} is complete.
\qed

 \section{Application and Examples}
 \label{sec-appl}

Theorem \ref{T1} gives a unified framework for a very large class of
SPDE driven by general L\'{e}vy noise, which  generalizes both the
classical results in \cite{KR79,Par75,PR07} and the recent results in
\cite{CM10,Liu+Roc}.  Within this framework,  the issue of the existence
and  uniqueness of solutions to a large class of  stochastic
evolution equations with monotone coefficients (cf. \cite{PR07,KR79}
for the stochastic porous medium equation and stochastic $p$-Laplace
equation) and with locally monotone coefficients (cf.
\cite{Liu+Roc,CM10} for stochastic  Burgers  type equations,
stochastic 2D Navier-Stokes equations and many other stochastic
hydrodynamical systems)  driven by more general L\'{e}vy processes
instead of Wiener processes can be treated.

For the simplicity of  notation  we use  $D_i$ to denote the spatial derivative $\frac{\partial}{\partial x_i}$, and
$\Lambda \subseteq \mathbb{R}^d$ is an open bounded domain with smooth
boundary.
 For the standard Sobolev space $W_0^{1,p}(\Lambda)$ $(p\ge 2)$  we always use the following (equivalent) Sobolev norm:
$$    \|u\|_{1,p}:=\left(\int_\Lambda |\nabla u(x)|^p d x\right)^{1/p}.    $$
For $d=2$, we recall
the following well-known estimate on $\mathbb{R}^2$ (cf. \cite{Temam_2001}\dela{{Liu+Roc}}):
\begin{equation}\label{2d}
  \|u\|_{L^4}^4 \le  C   \|u\|_{L^2}^2  \|\nabla u\|_{L^2}^2, \ u\in W_0^{1,2}(\Lambda) .
\end{equation}
We also  recall  the following estimate on $\mathbb{R}^3$ (cf. \cite{Temam_2001,MS02}):
\begin{equation}\label{e4}
  \|u\|_{L^4}^4 \le C  \|u\|_{L^2}  \|\nabla u\|_{L^2}^3, \  u\in W_0^{1,2}(\Lambda),
\end{equation}

We  first recall  the following  lemma in \cite{R13}, which partially generalizes the result in  \cite[Lemma 3.1]{Liu+Roc}.

\begin{lem}\label{L3.1}
Consider the Gelfand triple
$$ V:=W_0^{1,2}(\Lambda)\subset H:=L^2(\Lambda) \subset  \left(W^{1,2}_0(\Lambda)\right)^\ast = V^\ast.$$
and the operator
$$ A(u)=\Delta u+ \langle f(u),\nabla u\rangle,$$
where $f=(f_1,\dots,f_d):\mathbb{R}^d\to \mathbb{R}^d$ is a Lipschitz function and $\langle\; ,\;\rangle$ denotes inner product in $\R^d$.
Let $\text{Lip}(f)$ denote the corresponding Lipschitz constant.

$(1)$ If $d\le 4$, there exists $C\in]0,\infty[$ such that for all $u,v,w\in V$
$$\int_\Lambda |u||\nabla w| |v| \d x \le C \|u\|_V \|w\|_V \|v\|_V$$
In particular $A :V\to V^\ast$ is well defined. Furthermore, if $d=1$ or $f$ is bounded, $A$ satisfies $(H4)$ with $\alpha = 2$ and $\beta=2$ or $\beta=0$, respectively.

$(2)$ If $d=1$ or if each $f_i$ is bounded and $d=2$, then there exists $C\in (0,\infty)$ such that
$$2 { }_{V^*}\<A(u)-A(v), u-v\>_V    \le - \|u-v\|_V^2+  \left(C + C\|v\|_V^2  \right)\|u-v\|_H^2,\ u,v\in V.$$

$(3)$ If $f_i$ are bounded and independent of $u$ for $i=1,\cdots,d$, i.e.
$$ A(u)=\Delta u+ \langle f, \nabla u\rangle,  $$
then for any $d\ge 1$ we have
$$2 { }_{V^*}\<A(u)-A(v), u-v\>_V
   \le - \|u-v\|_V^2+  K\|u-v\|_H^2,\ u,v\in V.$$
\end{lem}
\begin{proof} The proof can be found in \cite{R13}, we include it here for the reader's convenience.

(1): We have for all $u,v\in V$
\begin{equation*}
\int_\Lambda \left| \langle f(u),\nabla u\rangle \right| |v| \d x \le \int_\Lambda \left(|f(0)| + \text{Lip}(f)|u|\right)|\nabla u||v| \d x.
\end{equation*}
To prove the first assertion, we note that for all $u,v,w\in V$
\begin{equation*}
\int_\Lambda |u||\nabla w||v| \d x \le \|uv\|_{L^2} \|w\|_V,
\end{equation*}
and by the generalized H\"older inequality the right hand side is dominated by
\begin{enumerate}
\item[(a)] $\|u\|_{L^2} \|v\|_{L^\infty} \|w\|_V,$
\item[(b)] $\|u\|_{L^4} \|v\|_{L^4} \|w\|_V,$
\item[(c)] $\|u\|_{L^d} \|v\|_{L^{\frac{2d}{d-2}}} \|w\|_V.$
\end{enumerate}
In the case $d=1, \; W^{1,2}_0(\Lambda) \subset L^\infty (\Lambda)$ continuously. Hence assertion (1) follows from (a) if $d=1$.

\noindent In the case $d=2,\; W^{1,2}_0(\Lambda) \subset L^p (\Lambda)$ continuously for all $p\in [1,\infty[$. Hence assertion (1)
follows from (b) if $d=2$.

\noindent In the case $d\le 3, W^{1,2}_0(\Lambda) \subset L^{\frac{2d}{d-2}} (\Lambda)$ continuously,
and $L^{\frac{2d}{d-2}} (\Lambda) \subset L^d (\Lambda)$ continuously if $d\le 4$. Hence assertion (1) follows from (c) if $d=3$ or $4$.

\noindent To prove the last part of the assertion we note that this is trivially true if $f$ is bounded.
If $d=1$ and if $f$ is merely Lipschitz continuous it follows immediately from (a).
\vspace*{10pt}

\noindent (2): We have 
\begin{equation*}
 \begin{split}
\label{e3.1}
& ~~~~ { }_{V^*}\<A(u)-A(v), u-v\>_V \\
 &= - \|u-v\|_V^2+ \sum_{i=1}^d \int_\Lambda \left(f_i(u)D_i u-f_i(v)D_i v\right)\left(u-v\right) \d  x.
\end{split}
\end{equation*}
To estimate the second term on the right hand side,
let $F_i :\R\to \R$ be such that $F_i(0)=0$ and $F_i'=f_i$ and $G_i : \R\to \R$ be such that $G_i(0)=0$ and $G_i'=F_i$.

Then
\begin{align*}
  & \int_\Lambda (f_i(u) D_i u-f_i (v) D_i v)(u-v) \d x\\
= & \int_\Lambda (f_i(u) D_i(u-v) +(f_i(u)-f_i(v))D_i v)\;(u-v) \d x\\
- & \int_\Lambda f_i(u-v) D_i(u-v) (u-v) \d x\\
+ & \int_\Lambda D_i\left(F_i(u-v)\right)(u-v) \d x,
\end{align*}
where integrating by parts and using that $u-v\in W^{1,2}_0 (\Lambda)$ we see that the last term on the right hand side is equal to
$$-\int_\Lambda D_i \left(G_i(u-v)\right) \d x,$$
which in turn after summation from $i=1$ to $d$ by Gauss's divergence theorem is zero,
since $G_i(u-v)=0$ on $\partial\Lambda$ for all $1\le i\le d$, because $u,v \in W^{1,2}_0(\Lambda)$.

Hence altogether we obtain
\begin{align}\label{5.1.14}
\nonumber & { }_{V^*}\<A(u)-A(v), u-v\>_V \\
 \le &-\|u-v\|^2_V + \int_\Lambda \<f(u)-f(u-v),\; \nabla(u-v)\>(u-v)d x\\
 \nonumber
& + \int_\Lambda \<f(u)-f(v),\;\nabla v\>(u-v)\d x
\end{align}
Now let us first consider the case $d=1$. Then using that $f$ is Lipschitz and applying Cauchy-Schwarz's and Young's inequalities we estimate,
the right hand side of \eqref{5.1.14} by
\begin{align}\label{5.1.14a}
-   & \|u-v\|_V^2 + \text{Lip}(f)\left(\|u-v\|_V \|v\|_{L^\infty} \|u-v\|_{L^2} + \|v\|_V \|u-v\|_{L^4}^2\right)\nonumber\\
\le -\frac34 & \|u-v\|_V^2 + C\left(\|v\|^2_V \|u-v\|_{L^2}^2 + \|v\|_V \|u-v\|_{L^4}^2 \right),
\end{align}
where $C\in (0,\infty)$ is independent of $u,v$ and we used that $W_0^{1,2}(\Lambda)\subset L^\infty(\Lambda)$ continuously, since $d=1$.

In the case $d=2$ and $f$ is bounded, we similarly obtain that the right hand side of \eqref{5.1.14} is dominated by
\begin{align}\label{5.1.14b}
-   & \|u-v\|_V^2 + 2\|f\|_{L^\infty} \|u-v\|_V \|u-v\|_{L^2} + \text{Lip}(f) \|v\|_V \|u-v\|_{L^4}^2\nonumber\\
\le -\frac{3}{4} & \|u-v\|_V^2 + C\left(\|u-v\|_{L^2}^2 + \|v\|_V \|u-v\|_{L^4}^2 \right),
\end{align}
where $C\in(0,\infty)$ is independent of $u,v$.

Hence combining (\ref{2d}) with \eqref{5.1.14a}, \eqref{5.1.14b}
and using Young's inequality we deduce that for some $C\in (0,\infty)$ 
$$ { }_{V^*}\<A(u)-A(v), u-v\>_V \le -\frac{1}{2} \|u-v\|_V^2+  \left(C  + C\|v\|_V^2  \right)\|u-v\|_H^2 \text{ for all } u,v\in V,$$
and assertion (2) is proved.
\vspace*{10pt}

\noindent $(3):$ In this case $A$ is a linear operator  and the assertion follows easily by the similar argument as in (2) (cf. also \cite{Liu+Roc}).
\end{proof}

For all examples presented in the remainder of  this section, we will only state the result on the existence and uniqueness of solutions. But we should remark that
one can also obtain those regularity estimates (\ref{e0}) and (\ref{e1}) by Theorem \ref{T1} if we do not have the large jumps term in
our equations (i.e. $g=0$).

\subsection{Semilinear type SPDEs}
\beg{exa}\label{exa-1}(Stochastic multidimensional Burgers type equations)
Let $\Lambda$ be an open bounded domain in $\mathbb{R}^d$ with smooth boundary. We consider
the following  semilinear stochastic  equation
 \begin{align}\begin{split}\label{rde}
 \d X_t=&\left(\Delta X_t+ \langle f(X_t),\nabla X_t\rangle+ f_0(X_t)\right)\d t+B(X_t)\d W_t\\
     &+\int_{D^c} h (X_{t-},z)\tilde{N}(\d t,\d z)+\int_D g( X_{t-},z)N(\d t, \d z) ;\\
X_0=& x.
\end{split}
  \end{align}

Suppose  the coefficients  satisfy
 the following conditions:

(i)  $f=(f_1, \cdots, f_d):\R\rightarrow\R^d$ is a  Lipschitz function;

(ii) $f_0$ is a continuous function  on $\mathbb{R}$ such that
\begin{equation}\label{c1}
\begin{split}
|f_0(x)| & \le C(\vert x \vert^r+1), \  x\in \mathbb{R};\\
 (f_0(x)-f_0(y))(x-y)&\le C(1+|y|^s)(x-y)^2, \  x,y\in \mathbb{R}.
\end{split}
     \end{equation}
where $C,r,s$ are some positive constants;

(iii) the function $B: W_0^{1,2}(\Lambda) \rightarrow \mathcal{T}_2(U; L^2(\Lambda))$ satisfies the following Lipschitz condition:
$$    \|B(v_1)-B(v_2)\|_{2}^2 \le  C \int_\Lambda |v_1-v_2|^2 \d x, \ v_1, v_2\in  W_0^{1,2}(\Lambda)  .    $$

(iv) $h, g: \mathbb{R}\times Z\rightarrow \mathbb{R}$ such that for all $v,v_1,v_2\in W^{1,2}_0(\Lambda)$,
\begin{align}\begin{split}\label{condition 1}
 &   \int_{D^c}\int_\Lambda |h(v_1,z)-h(v_2,z)|^2 \d x   \nu(\d z)
\le  C   \int_\Lambda |v_1-v_2|^2 \d x;  \\
&  \int_{D^c}\int_\Lambda |h(v,z)|^2 \d x  \nu(\d z)
\le C(1+\int_\Lambda |v|^2 \d x); \\
& \int_{D^c} \left(\int_\Lambda |h(v,z)|^2 \d x \right)^{3} \nu(\d z)
\leq C \left(1+ \left(\int_\Lambda |v|^2 \d x\right )^{3} \right).
\end{split}
\end{align}

Then we have the following result:

(1) If $d=1,r=3,s=2$, then for
 any $x \in L^{6}(\Omega, \mathcal{F}_0,\mathbb{P};H)$,
    $(\ref{rde})$
    has a unique solution $\{X_t\}_{t\in [0,T]}$.

(2) If $d=2,r=\frac{7}{3},s=2$ and each $f_i$ is bounded, then for
 any $x\in L^{6}(\Omega, \mathcal{F}_0,\mathbb{P};H)$,
    $(\ref{rde})$
    has a unique solution $\{X_t\}_{t\in [0,T]}$.

(3) If $d= 3,r=\frac{7}{3},s=\frac{4}{3}$ and each $f_i$ is
bounded measurable function which is independent of $X_t$,
 then for
 any $x\in L^{6}(\Omega, \mathcal{F}_0,\mathbb{P};H)$,
    $(\ref{rde})$
    has a unique solution $\{X_t\}_{t\in [0,T]}$.
\end{exa}

\begin{proof}
We consider
the following Gelfand triple
$$ V:=W^{1,2}_0(\Lambda) \subseteq H:=L^2(\Lambda)\subseteq (W^{1,2}_0(\Lambda))^*=V^*$$
and define the operator
$$ A(u)=\Delta u+ \langle f(u), \nabla  u\rangle+ f_0(u), \ u\in V.  $$
By Lemma \ref{L3.1}, one can show that $A, B$ satisfies $(H1)$-$(H4)$ with $\alpha=2, \beta=4$ (see \cite[Example 3.2]{Liu+Roc}).

Moreover, it is easy to show that  $h$ also satisfies the required conditions (i.e. $(H2)$, \eqref{c3} and \eqref{c4}) by (\ref{condition 1}).

Then  all assertions follow from Theorem \ref{T1}.
\end{proof}

\begin{rem}
(1) If $d=1$ and $f(x)=x$, Theorem \ref{T1}
can be applied to
 classical  stochastic Burgers equation (i.e. (\ref{rde}) with $f_0\equiv0$).
 Therefore, the above example improves the main result in
\cite{DX1} (Theorem 2.2) in the sense that  we allow  the coefficient $B$ in front of Wiener noise
 to be non-additive type.
 Another improvement is that we also allow a polynomial  perturbation term $f_0$ in
the drift of (\ref{rde}).
For example,  one can take $f_0(x)=-x^3+c_1x^2+c_2x~ (c_1,c_2\in \R)$ and show that \eqref{c1} holds.
 Hence
(\ref{rde}) also covers some stochastic reaction-diffusion type equations driven
by certain type of a L\'{e}vy noise (cf. \cite{Brz+Haus_2009}).

(2)  If $Z=\R^d$, $D^c=\{z\in\R^d:  |z|\le 1  \}$ and $\nu$ is a L\'{e}vy measure on $\R^d$,
then one simple sufficient condition for  $h$ satisfying (\ref{condition 1}) is to assume
\begin{align*}
 &  |h(x,z)-h(y,z)|  \le C |x-y|  |z|, \ x,y\in \R, \ z\in D^c;       \\
  &  |h(x,z)|  \le C (1+\vert x \vert)  |z|, \ x,y\in \R, \ z\in D^c.
\end{align*}

(3) One should note that in the Example \ref{exa-1},
$B$ is assumed to be Lipschitz from $W^{1,2}_0(\Lambda)$ (w.r.t. $\|\cdot\|_H$)
to $\mathcal{T}_2(U; L^2(\Lambda))$ only for simplicity.
Actually, the Lipschitz condition on $B$ can even be weakened to the requirement
$$    \|B(v_1)-B(v_2)\|_{2}^2 \le \|v_1-v_2\|_{V}^2  +\left(K+ K\|v_2\|_V^2\right)\|v_1-v_2\|_H^2.          $$
\end{rem}

\subsection{Quasi-linear type SPDEs}

Besides from the example of semilinear SPDE above,  we can also apply the main result to the following  quasi-linear SPDE  on
$\mathbb{R}^d\ (d\ge 3)$ driven by L\'{e}vy noise.

\begin{exa}(Stochastic $p$-Laplace equations) We consider
the following equation on $\mathbb{R}^d$ for $p> 2$
\begin{align}\begin{split}\label{p-Laplace}
 \d X_t=&\left(\sum_{i=1}^d D_i\left(|D_i X_t|^{p-2} D_i X_t  \right)+ f_0(X_t)\right)\d t+B(X_t)\d W_t\\
     &+\int_{D^c}f(X_{t-},z)\tilde{N}(\d t,\d z)+\int_D g( X_{t-},z)N(\d t, \d z) ;\\
X_0=& x.
\end{split}
  \end{align}

Suppose the following conditions hold:

(i) $f_0$ is a continuous function  on $\mathbb{R}$ such that
\begin{equation}
\begin{split}\label{c7}
 f_0(x)x  &\le C(\vert x \vert^{\frac{p}{2}+1}+1), \  x\in \mathbb{R};\\
|f_0(x)| & \le C(\vert x \vert^{r}+1), \  x\in \mathbb{R};\\
(f_0(x)-f_0(y))(x-y)&\le C(1+|y|^t)|x-y|^{s}, \  x,y\in \mathbb{R},
\end{split}
\end{equation}
where  $C>0$ and $r,s,t\ge 1$ are some constants.

(ii) $B: W^{1,p}_0(\Lambda)\rightarrow \mathcal{T}_2(U; L^2(\Lambda))$  satisfies the following condition:
$$    \|B(v_1)-B(v_2)\|_{2}^2 \le  C \int_\Lambda |v_1-v_2|^2 \d x, \ v_1, v_2\in  W_0^{1,p}(\Lambda)  .    $$

(iv) $f, g: \mathbb{R}\times Z\rightarrow \mathbb{R}$ such that for all $v,v_1,v_2\in W^{1,p}_0(\Lambda)$,
\begin{align}\begin{split}\label{condition 11}
 &   \int_{D^c}\int_\Lambda |f(v_1,z)-f(v_2,z)|^2 \d x   \nu(\d z)
 \le  C   \int_\Lambda |v_1-v_2|^2 \d x ;\\
&  \int_{D^c}\int_\Lambda |f(v,z)|^2 \d x  \nu(\d z)
\le C(1+\int_\Lambda |v|^2 \d x); \\
& \int_{D^c} \left(\int_\Lambda |f(v,z)|^2 \d x \right)^{3} \nu(\d z)
\leq C \left(1+ \left(\int_\Lambda |v|^2 \d x\right )^{3} \right).
\end{split} \end{align}

Then we have

(1) if $d<p$, $s=2$, $r= p+1$ and $t\le p$, then for any $x\in  L^{6}(\Omega, \mathcal{F}_0,\mathbb{P}; H)$,  $(\ref{p-Laplace})$ has a unique solution.

(2) if $d>p$, $2<s<p$, $r=\frac{2p}{d}+p-1$ and $t\le  \min\left\{\frac{p^2(s-2)}{(d-p)(p-2)},  \frac{p(p-s)}{p-2}\right\} $,
for any $x\in  L^{6}(\Omega, \mathcal{F}_0,\mathbb{P}; H)$
  $(\ref{p-Laplace})$ has a unique solution.
\end{exa}

\begin{proof}
(1) We consider
the following Gelfand triple ($q:=\frac{p}{p-1}$)
$$ V:=W_0^{1,p}(\Lambda)\subseteq H:=L^2(\Lambda) \subseteq  W^{-1,q}(\Lambda)=V^*. $$
 It is well known that $\sum_{i=1}^d D_i\left(|D_iu|^{p-2} D_i u\right) $
satisfy $(H1)$-$(H4)$ with $\alpha=p$ (cf. \cite{L08b}).  In
particular,
 there exists a constant $\delta>0$ such that
\begin{equation}\label{e7}
 \sum_{i=1}^d {~}_{V^*}\<  D_i\left(|D_iu|^{p-2} D_i u\right)-  D_i\left(|D_iv|^{p-2} D_i v\right)  , u-v  \>_V
\le - \delta \|u-v\|_V^p, \ u,v\in  W_0^{1,p}(\Lambda).
\end{equation}

Recall that  for $d<p$ we have the following Sobolev embedding
$$      W_0^{1,p}(\Lambda) \subseteq   L^{\infty}(\Lambda).   $$
Hence by (\ref{c7}) we have
\begin{equation}
\begin{split}
  {~}_{V^*}\<f_0(u)-f_0(v),u-v\>_V & \le C \int_\Lambda \left(1+ |v|^t \right) |u-v|^2 d x\\
  &\le C \left( 1+ \|v\|_{L^{\infty}}^t   \right) \|u-v\|^2_{L^{2}}\\
  &\le C \left( 1+ \|v\|_{V}^t   \right) \|u-v\|_{H}^{2}, \ u,v\in V,
\end{split}
\end{equation}
where $C$ is a constant may change from line to line.

Hence $(H2)$ holds with $\rho(v)=C\|v\|_V^t$.

Note that from (\ref{c7}) we have
\begin{equation}\label{e51}
\begin{split}
{~}_{V^*}\<f_0(u), u\>_V & \le C \int_\Lambda (1+ |u|^{\frac{p}{2}+1})dx\\
       & \le C\left(1+\|u\|_{L^\infty}^{p/2}\|u\|_H\right)\\
   & \le \frac{\delta}{2} \|u\|_V^p + C\left(1+\|u\|_H^2\right), \ u\in V.
\end{split}
\end{equation}
Therefore, \eqref{e51} together with \eqref{e7} verify $(H3)$ with $\alpha=p$.

$(H4)$ with $\beta=4$ (in fact one may take $\beta=\frac{2p}{p-1}<4$)  follows from the following estimate:
$$  \|f_0(u)\|_{V^*} \le C\left(1+\|u\|_{L^{p+1}}^{p+1}\right)\le C\left(1+\|u\|_{L^{\infty}}^{p-1}\|u\|_{H}^2\right)
\le  C\left(1+\|u\|_{V}^{p-1}\|u\|_{H}^2\right)
, \ u\in V. $$

Then combining with (\ref{condition 11}) we know that  the assertions follow from Theorem \ref{T1}.

(2)
Note that  for $d>p$ we have the following Sobolev embedding
$$      W_0^{1,p}(\Lambda) \subseteq   L^{p_0}(\Lambda), \ p_0=\frac{dp}{d-p}.   $$
Let $t_0=\frac{p(s-2)}{s(p-2)}\in(0,1)$ and $p_1\in(2,p_0)$ such that
$$     \frac{1}{p_1}=\frac{1-t_0}{2}+ \frac{t_0}{p_0}.    $$
Then we have the following interpolation inequality:
$$   \|u\|_{L^{p_1}}\le   \|u\|_{L^{2}}^{1-t_0}   \|u\|_{L^{p_0}}^{t_0}, \ u\in W_0^{1,p}(\Lambda).   $$
Since $2<s<p$, it is easy to show that $s<p_1$.

Let $p_2=\frac{p_1}{p_1-s}$, then by  (\ref{c7}) we have
\begin{equation}\label{e8}
\begin{split}
  {~}_{V^*}\<f_0(u)-f_0(v),u-v\>_V & \le C \int_\Lambda \left(1+ |v|^t \right) |u-v|^s d x\\
  &\le C \left( 1 +\|v\|_{L^{tp_2}}^t   \right) \|u-v\|^s_{L^{p_1}}\\
  &\le C \left( 1+\|v\|_{L^{tp_2}}^t   \right) \|u-v\|_{L^{2}}^{s(1-t_0)}   \|u-v\|_{L^{p_0}}^{st_0}\\
&\le \varepsilon \|u-v\|_{L^{p_0}}^{p} +
C_\varepsilon \left( 1+\|v\|_{L^{tp_2}}^{tb}   \right)   \|u-v\|_{L^{2}}^{2}
,
\end{split}
\end{equation}
where $\varepsilon, C_\varepsilon$ are some constants and the last step follows from the following Young inequality
$$  xy\le \varepsilon x^a +C_\varepsilon y^b, \ x,y\in\mathbb{R},\  a=\frac{p-2}{s-2},\ b=\frac{p-2}{p-s}. $$
With some calculations, one have
$$  \frac{s}{p_1}=\frac{p-s}{p-2}+\frac{p(s-2)}{p_0(p-2)},\ p_2=\frac{p_0(p-2)}{(p_0-p)(s-2)}. $$
Hence if $t\le \frac{(p_0-p)(s-2)}{p-2}$, then
$$ \|u\|_{L^{tp_2}}\le C \|u\|_{L^{p_0}} \le C  \|u\|_{V}, \ v\in V.   $$
Therefore, $(H2)$ follows from (\ref{e7}) and (\ref{e8}).

$(H3)$ can be verified for $\alpha=p$ in a similar manner.

For $r=\frac{2p}{d}+p-1$, by the interpolation inequality we have
$$\|f_0(u)\|_{V^*}\le C\left(1+ \|u\|_{L^{rp_0^\prime}}^r \right)
\le C\left( 1+\|u\|_{p_0}^{p-1}\|u\|_H^{\theta} \right), \ u\in V, $$
where
$$ \frac{1}{p_0}+\frac{1}{p_0^\prime}=1,\ \ \theta=\frac{2p}{d}.  $$
Therefore, $(H4)$ also holds with $\beta=4$.

Then all assertions  follow from Theorem \ref{T1}.
\end{proof}

\begin{rem}
One further generalization is to replace  $\sum_{i=1}^d
D_i\left(|D_iu|^{p-2} D_i u\right)$ by more general quasi-linear
differential operator
$$ \sum_{|\alpha|\le m} (-1)^{|\alpha|}D_\alpha A_\alpha(x,Du(x,t);t),  $$
where $Du=(D_\beta u)_{|\beta|\le m}$. Under certain assumptions (cf. e.g.\cite[Proposition 30.10]{Z90}) this operator
also satisfies the monotonicity and coercivity conditions.
Then by a similar argument,
according to Theorem \ref{T1}, we can  obtain the existence and
uniqueness of solutions to this type of quasi-linear SPDE driven by L\'{e}vy noise.
\end{rem}

\subsection{Stochastic hydrodynamical systems}
The next example is  the  stochastic 2D Navier-Stokes equation driven by L\'{e}vy noise (cf. \cite{BCF4,Flandoli_G_1995,MS02,Liu+Roc}
for Wiener noise case).
The classical Navier-Stokes equation is a very important model in fluid mechanics to describe the time evolution
of incompressible fluids, it can be formulated as follows (2D case):
\begin{equation*}
 \begin{split}
 & \partial_t u(t)=\nu \Delta u(t)- \left( u(t) \cdot \nabla \right) u(t) -\nabla p(t)+f(t), \\
& \nabla \cdot u(t)=0,
 \end{split}
\end{equation*}
where  $u(t, x) = (u^1(t, x), u^2(t, x))$  represents the velocity field, $\nu$ is the
viscosity constant (we keep the standard notation and it should cause no confusion with
the measure $\nu$ corresponding to the L\'{e}vy process),  $p(t, x)$ denotes the pressure and $f$ is an external force field acting
on the fluid.

Let $\Lambda$ be a bounded domain in $\mathbb{R}^2$ with smooth boundary. Define
$$ V=\left\{ v\in W_0^{1,2}(\Lambda,\mathbb{R}^2): \nabla \cdot v=0 \  a.e.\  \text{in} \ \Lambda   \right\}, \
\|v\|_V:=\left(\int_\Lambda |\nabla v|^2 dx  \right)^{1/2},
$$
and $H$ is the closure of $V$ in the following norm
$$ \|v\|_H:=\left(\int_\Lambda | v|^2 dx  \right)^{1/2}.$$
The linear operator $P_H$ (the Helmholtz-Leray projection) and $A$ (Stokes operator with viscosity constant
$\nu$) are defined by
$$ P_H: L^2(\Lambda, \mathbb{R}^2)\rightarrow H\ \  \text{ orthogonal projection}; $$
$$  A: W^{2,2}(\Lambda, \mathbb{R}^2)\cap V\rightarrow H, \ Au=\nu P_H \Delta u .  $$
It is well known that  the Navier-Stokes equation can be
reformulated as follows:
\begin{equation}\label{NSE}
u'=Au+F(u)+f_0,\ u(0)=u_0\in H,
\end{equation}
where $f_0\in L^2(0,T;V^*)$ denotes some external force and
$$ F:  \mathcal{D}_F\subset H\times V\rightarrow H, \ F(u,v)=- P_H\left[\left(u \cdot \nabla\right) v\right],
F(u)=F(u,u).  $$
It is standard that in the framework of  the Gelfand triple
$$     V\subseteq H\equiv H^*\subseteq V^*,   $$
one can show that the following mappings
$$ A: V\rightarrow V^*, \  F: V\times V\rightarrow V^*  $$
are well defined. In particular, we have
$$ { }_{V^*}\<F(u,v),w\>_V=-{ }_{V^*}\<F(u,w),v\>_V, \   { }_{V^*}\<F(u,v),v\>_V=0,\  u,v,w\in V. $$
Now we consider the stochastic 2D Navier-Stokes equation driven by L\'{e}vy noise:
\begin{equation}\begin{split}\label{SNSE}
\d X_t=&\left(AX_t+F(X_t)+f_0(t)\right)\d t+ B(X_t)\d W_t\\
         &~~+\int_{D^c}f(X_{t-},z)\tilde{N}(\d t,\d z)+\int_D g(X_{t-},z)N(\d t, \d z); \\
X_0=& x.
\end{split}
\end{equation}

\begin{exa}(Stochastic 2D Navier-Stokes equation)
Suppose that
$B:V\rightarrow \mathcal{T}_2(U; H)$  and $f,g: \mathbb{R}\times Z\rightarrow \mathbb{R}$
satisfy the following conditions:
\begin{align}\begin{split}\label{NS}
 &  \|B(v_1)-B(v_2)\|_2^2+  \int_{D^c}\|f(v_1,z)-f(v_2,z)\|^2_H\nu(\d z)  \le  C \|v_1-v_2\|_H^2;\\
&  \int_{D^c}\|f(v,z)\|_H^2\nu(\d z)  \le C(1+\|v\|_H^2); \\
& \int_{D^c}   \|f(v,z)\|^4_H\nu(\d z)\leq C (1+\|v\|^4_H),
\end{split}
 \end{align}
 where $C$ is some constant.

 Then for any $x \in L^{4}(\Omega, \mathcal{F}_0,\mathbb{P}; H)$,
    $(\ref{SNSE})$
    has a unique solution $\{X_t\}_{t\in [0,T]}$.
\end{exa}

\begin{proof} The hemicontinuity $(H1)$ is obvious since $A$ is linear and $F$ is bilinear.

Note that $ { }_{V^*}\<F(v),v\>_V=0$, it is also easy to show that $(H3)$ holds
with $\alpha=2$:
\begin{equation*}\begin{split}
 { }_{V^*}\<Av+F(v)+f_0(t),v\>_V &\le -\nu\|v\|_V^2+\|f_0(t)\|_{V^*}\|v\|_V \le
  -\frac{\nu}{2}\|v\|_V^2+C\|f_0(t)\|_{V^*}^2, \ v\in V,\\
  \|B(v)\|_2^2 &\le 2K\|v\|_H^2+2\|B(0)\|_2^2, \ v\in V.
\end{split}
\end{equation*}
 Recall the following estimates (cf.
e.g.\cite[Lemmas 2.1, 2.2]{MS02})
\begin{equation}\label{e2}
 \begin{split}
  |{ }_{V^*}\<F(w),v\>_V| &\le 2 \|w\|_{L^4(\Lambda;\mathbb{R}^2)}\|v\|_V; \\
  |{ }_{V^*}\<F(w),v\>_V| &\le 2  \|w\|_V^{3/2} \|w\|_H^{1/2}  \|v\|_{L^4(\Lambda;\mathbb{R}^2)}, v,w\in V. \\
 \end{split}
\end{equation}
Then we have
\begin{equation}
 \begin{split}
  { }_{V^*}\<F(u)-F(v),u-v\>_V &=-  { }_{V^*}\<F(u,u-v),v\>_V+  { }_{V^*}\<F(v,u-v),v\>_V \\
 &= -  { }_{V^*}\<F(u-v),v\>_V \\
 &\le 2  \|u-v\|_V^{3/2} \|u-v\|_H^{1/2}  \|v\|_{L^4(\Lambda;\mathbb{R}^2)} \\
& \le \frac{\nu}{2} \|u-v\|_V^{2} + \frac{32}{\nu^3}  \|v\|_{L^4(\Lambda;\mathbb{R}^2)}^4 \|u-v\|_H^{2},
\ u,v\in V.
 \end{split}
\end{equation}
Hence we have the local monotonicity:
$$  { }_{V^*}\<Au+F(u)-Av-F(v),u-v\>_V
\le -\frac{\nu}{2} \|u-v\|_V^{2} + \frac{32}{\nu^3}  \|v\|_{L^4(\Lambda;\mathbb{R}^2)}^4 \|u-v\|_H^{2}.  $$
Combining with (\ref{NS}) we know that $(H2)$ holds  with $\rho(v)=C
\|v\|_{L^4(\Lambda;\mathbb{R}^2)}^4 $.

(\ref{e2}) and (\ref{2d}) imply that  $(H4)$ holds
with $\beta=2$.

Then it is easy to see that  the existence and uniqueness of solutions to (\ref{SNSE})
follows from Theorem \ref{T1}.
\end{proof}

\begin{rem}\label{rem4.1}
As we mentioned in the introduction, besides the stochastic 2D Navier-Stokes equation,
 many other hydrodynamical systems also satisfy the local monotonicity condition $(H2)$ and coercivity condition $(H3)$.
For example, in a recent work of Chueshov and Millet \cite{CM10}, they have studied the well-posedness and large deviation principle for an abstract stochastic
semilinear equation (driven by Wiener noise) which covers a wide class of fluid dynamical models. In fact, the Condition (C1) and (C2) in \cite{CM10}
 implies that the assumptions in
Theorem \ref{T1} hold.  More precisely,   (2.2) in \cite{CM10} implies the coercivity  $(H3)$ holds, and the local monotonicity  $(H2)$
follows   from (2.4) (or (2.8)) in \cite{CM10}.  Other assumptions in Theorem \ref{T1} can be also verified easily.

Therefore, Theorem \ref{T1} can be applied to show the well-posedness of all hydrodynamical models in \cite{CM10} driven by general L\'{e}vy noise instead of Wiener noise,
e.g.  stochastic
magneto-hydrodynamic equations,  stochastic Boussinesq model for
the B\'{e}nard convection, stochastic 2D magnetic B\'{e}nard problem and  stochastic 3D Leray-$\alpha$ model driven by L\'{e}vy noise.
\end{rem}

\subsection{Stochastic power law fluids}

The next example of SPDE is a model which describes the velocity field of a viscous and
 incompressible non-Newtonian fluid subject to some random forcing.
The deterministic model has been studied intensively in PDE theory (cf.\cite{FR,MNRR} and the references therein).
Let $\Lambda$ be a  bounded domain in $\mathbb{R}^d~ (d\ge 2)$  with  smooth boundary.
For a vector field    $u: \Lambda\rightarrow \R^d$,
we denote the rate of strain tensor by
$$ e(u): \Lambda\rightarrow \R^d\otimes \R^d; \   e_{i,j} (u)=\frac{\partial_i u_j+ \partial_j u_i}{2},
\ i,j=1,\cdots, d.  $$
Now we consider the case  that the extra stress tensor has the following polynomial form:
$$ \tau(u): \Lambda \rightarrow\R^d\otimes\R^d; \ \tau(u)=2\nu (1+|e(u)|)^{p-2} e(u),   $$
where   $\nu>0$ is the kinematic viscosity and $p>1$ is some constant.

In the case of deterministic forcing,  the dynamics of  power law fluids can be modeled   by the following PDE (cf.\cite[Chapter 5]{MNRR}):
\begin{equation}\label{power law fluids}
 \begin{split}
& \partial_t u= \text{div} \left(\tau(u) \right)- (u\cdot \nabla)u-\nabla p +f , \\
&  \text{div}  (u)=0, \  u|_{\partial\Lambda}=0,     ~ u(0)=u_0,
 \end{split}
\end{equation}
where $u=u(t,x)=\left( u_i(t,x) \right)_{i=1}^d$ is the velocity field, $p$ is the pressure,
$f$ is some external force and
$$ u\cdot \nabla=\sum_{j=1}^d u_j \partial_j, \  \
 \text{div}\left( \tau (u) \right)= \left( \sum_{j=1}^d \partial_j \tau_{i,j} (u)  \right)_{i=1}^d.  $$
\begin{rem}
(1) Note that $p=2$ describes the Newtonian fluids and $(\ref{power law fluids})$ reduces to  the classical
Navier-Stokes equation.

(2) The shear shining fluids (i.e. $p\in(1,2)$) and the shear thickening fluids (i.e.  $p\in (2, \infty)$) has been also widely
studied in different fields of science and engineering (cf. \cite{FR,MNRR}).
\end{rem}

Now we consider the following Gelfand triple
$$ V\subset H \subset V^*,   $$
where
\begin{align*}
 V&=\left\{ u\in W_0^{1,p}(\Lambda; \R^d):\  \nabla\cdot u=0 \  a.e. \  \text{in} \ \Lambda    \right\}; \\
  H&=\left\{ u\in L^2(\Lambda; \R^d):\   \nabla\cdot u=0 \ a.e. \  \text{in} \ \Lambda, \ u\cdot n=0 \
  \text{on} \ \partial\Lambda     \right\}.
\end{align*}
Let $P_H$ be the orthogonal (Helmhotz-Leray) projection from
 $L^2(\Lambda,\R^d)$ to $H$. Similarly as in the previous example, we can show that
 the following operators
$$  A: W^{2,p}(\Lambda; \R^d)\cap V \rightarrow H, \ A(u):= P_H \left[  \text{div}( \tau( u ) ) \right] ; $$
$$    F: W^{2,p}(\Lambda; \R^d)\cap V \times W^{2,p}(\Lambda; \R^d)\cap V  \rightarrow H; \  F(u, v):=- P_H \left[ (u\cdot \nabla) v   \right],  \ F(u):=F(u,u)  $$
can be extended to the   well defined operators:
$$     A: V\rightarrow V^*; \ F: V\times V\rightarrow V^*.              $$
In particular, one can show that
$$     { }_{V^*}\<A(u),  v\>_V = -  \int_\Lambda \sum_{i,j=1}^d  \tau_{i,j}(u) e_{i,j}(v) \d x, \  u,v\in V;     $$
$$       { }_{V^*}\<F(u,v),  w\>_V=  -   { }_{V^*}\<F(u,w),  v\>_V, \   { }_{V^*}\<F(u,v),  v\>_V=0, \  u,v,w\in V.         $$

Now we consider  stochastic  equation of power law fluids driven by L\'{e}vy noise:
\begin{equation}\begin{split}\label{PLF}
\d X_t=&\left(AX_t+F(X_t)+ f_0(t)\right)\d t+ B(X_t)\d W_t\\
         &~~+\int_{D^c}f(X_{t-},z)\tilde{N}(\d t,\d z)+\int_D g(X_{t-},z)N(\d t, \d z); \\
X_0=& x,
\end{split}
\end{equation}
 where $f_0:=P_H f$.

\begin{exa}(Stochastic  equation of power law fluids)
Suppose that $f_0\in L^2([0,T]; H)$,
$B:V\rightarrow \mathcal{T}_2(U; H)$  and $f,g: \mathbb{R}\times Z\rightarrow \mathbb{R}$
satisfy the following conditions:
\begin{align}\begin{split}\label{}
 &  \|B(v_1)-B(v_2)\|_2^2+  \int_{D^c}\|f(v_1,z)-f(v_2,z)\|^2_H\nu(\d z)  \le  C \|v_1-v_2\|_H^2;\\
&  \int_{D^c}\|f(v,z)\|_H^2\nu(\d z)  \le C(1+\|v\|_H^2); \\
& \int_{D^c}   \|f(v,z)\|^4_H\nu(\d z)\leq C (1+\|v\|^4_H),
\end{split}
 \end{align}
 where $C$ is some constant.

 Then if  $p\ge \frac{d+2}{2}$, for any $x \in L^{4}(\Omega, \mathcal{F}_0,\mathbb{P}; H)$
    $(\ref{PLF})$
    has a unique solution $\{X_t\}_{t\in [0,T]}$.
\end{exa}

\begin{proof} Without loss of generality we may assume the viscosity constant $\nu=1$.

We first recall the well known Korn's inequality for $p\in (1, \infty)$ (cf. \cite[Theorem 1.10 (pp.196)]{MNRR}):
$$  \int_\Lambda |e(u)|^p \d x \ge C_p \|u\|_{1,p}, \  u\in W_0^{1,p}(\Lambda; \R^d),  $$
where  $C_p>0$ is  some constant.

The following inequalities are also used very often in the study of power law fluids (cf. \cite[pp.198 Lemma 1.19]{MNRR}):
\begin{equation}\label{e11}
 \begin{split}
 & |\tau_{i,j} (u) | \le C(1+|e(u)|)^{p-1}, \ i,j=1, \cdots, d; \\
&  \sum_{i,j=1}^d \tau_{i,j}(u) e_{i,j} (u)\ge C(|e(u)|^p-1); \\
&   \sum_{i,j=1}^d  (\tau_{i,j} (u)-\tau_{i,j} (v))(e_{i,j}(u)- e_{i,j}(v))\ge C\left(|e(u)-e(v)|^2 +|e(u)-e(v)|^p  \right).
\end{split}
\end{equation}
Then by the interpolation inequality  and Young's inequality  one can show  that
\begin{equation*}
 \begin{split}
 & ~~ { }_{V^*}\<F(u)-F(v),u-v\>_V\\
 &= -  { }_{V^*}\<F(u-v), v\>_V \\
&=  { }_{V^*}\<F(u-v,v), u-v\>_V \\
&\le C   \|v\|_{V}  \|u-v\|_{L^{\frac{2p}{p-1}}}^2  \\
 &\le C  \|v\|_{V}  \|u-v\|_{1,2}^{\frac{d}{p}} \|u-v\|_H^{\frac{2p-d}{p}}  \\
& \le \varepsilon \|u-v\|_{1,2}^{2} + C_\varepsilon \|v\|_{V}^{\frac{2p}{2p-d}} \|u-v\|_H^{2},
\  u,v\in V.
 \end{split}
\end{equation*}
By (\ref{e11}) and Korn's inequality  we have
\begin{equation*}
 \begin{split}
 &  { }_{V^*}\<A(u)-A(v), u-v\>_V\\
=&- \int_\Lambda \sum_{i,j=1}^d
\left( \tau_{i,j}(u) -\tau_{i,j}(v)\right) \left( e_{i,j}(u)-e_{i,j}(v) \right) \d x\\
\le& -C\|e(u)-e(v)\|_H^2\\
\le& -C\|u-v\|_{ 1,2}^2.
 \end{split}
\end{equation*}
Hence we have the following estimate:
$$  { }_{V^*}\<A(u)+F(u)-A(v)-F(v), u-v\>_V
 \le -(C-\varepsilon) \|u-v\|_{ 1,2 }^{2} + C_\varepsilon \|v\|_{V}^{\frac{2p}{2p-d}} \|u-v\|_H^2, $$
 where $\varepsilon>0$ and $C_\varepsilon$ are some constants.

Hence  $(H2)$ holds with $\rho(v)=C_\varepsilon \|v\|_{V}^{\frac{2p}{2p-d}}$.

It is also easy to  verify $(H3)$ with $\alpha=p$ as follows:
$$  { }_{V^*}\<A(v)+F(v),  v\>_V\le -C_1\int_\Lambda |e(v)|^p \d x+ C_2\le -C_3\|v\|_V^p+C_2,   $$
where $C_1, C_2, C_3$ are some constants.

Note that
$$
 \left| { }_{V^*}\<F(v),u\>_V \right|
= \left|  { }_{V^*}\<F(v,u), v\>_V \right| \le   \|u\|_{V}  \|v\|_{L^{\frac{2p}{p-1}}}^2,
\ u,v\in V,
 $$
hence we have
$$ \|F(v)\|_{V^*}\le   \|v\|_{L^{\frac{2p}{p-1}}}^2, \ v\in V.$$
Then by the interpolation inequality and Sobolev's inequality we have
$$    \|v\|_{L^{\frac{2p}{p-1}}}\le \|v\|_{L^q}^\gamma \|v\|_{L^2}^{1-\gamma} \le C \|v\|_V^\gamma\|v\|_H^{1-\gamma},    $$
where $q=\frac{dp}{d-p}$ and $\gamma=\frac{d}{(d+2)p-2d}$.

Note that $2\gamma\le p-1$ if $p\ge  \frac{2+d}{2}$, and it is also easy to see that
$$ \|A(v)\|_{V^*} \le C (1+\|v\|_V^{p-1}), \ v\in V. $$
 Hence  the growth condition $(H4)$ also holds.

Moreover, note that $d\ge \frac{2+d}{2}$, it is easy to show that (\ref{c5}) also holds.

Therefore,   the existence and uniqueness of solutions to (\ref{PLF})
follows from Theorem \ref{T1}.
\end{proof}

\begin{rem} In \cite{TY11} the authors established the existence and uniqueness of weak solutions for $(\ref{PLF})$
with additive Wiener noise. They first  considered  the Galerkin approximation and  showed the tightness of the distributions
 of the corresponding approximating solutions.
Then they proved that the limit is a weak solution of $(\ref{PLF})$ with additive Wiener noise.

In \cite{LR13} the authors obtained the existence and uniqueness of strong solutions for $(\ref{PLF})$
with multiplicative Wiener noise.
Here by applying Theorem \ref{T1} we
establish the existence and uniqueness of strong solutions  to $(\ref{PLF})$ in $L^2$-space of divergence free vector fields
with multiplicative L\'{e}vy noise.
 \end{rem}


\end{document}